\documentclass[11pt,a4paper]{amsart}

\usepackage[utf8]{inputenc}

\usepackage{amsmath}
\usepackage{amssymb}
\usepackage{stmaryrd}
\usepackage{dsfont}
\usepackage{amsthm}
\usepackage{graphicx}
\usepackage{xcolor}
\usepackage[all]{xy}
\usepackage{comment}
\usepackage{url}
\usepackage[foot]{amsaddr}

\newcommand{\cut}[0]{c}
\newcommand{\lcut}[0]{\rule{0.4pt}{1ex}\hspace{1pt}\cut}
\newcommand{\rcut}[0]{\cut\hspace{1pt}\rule{0.4pt}{1ex}}
\newcommand{\bcut}[0]{\underline\cut}
\newcommand{\tcut}[0]{\overline\cut}

\newcommand{\hcut}[0]{h\cut}

\newcommand{\vcut}[0]{v\cut}

\newcommand{\Fb}[0]{B}
\newcommand{\Fd}[0]{{\mathbf{d}}}
\newcommand{\Fh}[0]{{\mathbf{h}}}
\newcommand{\Fv}[0]{{\mathbf{v}}}
\newcommand{\Fs}[0]{B}
\newcommand{\Rs}[0]{R}

\newcommand{\dist}{\mathrm{d}}
\newcommand{\disti}{\dist_I}
\newcommand{\distb}{\dist_b}
\newcommand{\barcode}{\mathcal{B}}
\newcommand{\barcodespace}{\overline{\mathcal{B}}}

\newcommand{\op}{\mathrm{op}}
\newcommand{\M}{M}
\newcommand{\MRs}{M}

\newcommand{\field}{\mathbf{k}}

\newcommand{\sR}[0]{\mathbb{R}}
\newcommand{\sQ}[0]{\mathbb{Q}}

\newcommand{\sZ}[0]{\mathbb{Z}}
\newcommand{\sX}[0]{\mathbb{X}}
\newcommand{\sN}[0]{\mathbb{N}}
\newcommand{\sU}[0]{\mathbb{U}}
\newcommand{\cV}[0]{\mathcal{V}}
\newcommand{\cF}[0]{\mathcal{F}}

\newcommand{\cG}[0]{\mathcal{G}}
\newcommand{\cS}[0]{\mathcal{S}}

\newcommand{\Hgr}[0]{\mathsf{H}}
\newcommand{\Ran}{\mathsf{Ran}}

\newcommand{\supp}[0]{\mathrm{supp}}

\DeclareMathOperator{\Ima}{Im}
\DeclareMathOperator{\Ker}{Ker}

\newcommand{\e}{\varepsilon}
\newcommand{\func}{f}

\theoremstyle{plain}
\newtheorem{thm}{Theorem}[section]
\newtheorem{lem}[thm]{Lemma}
\newtheorem{prop}[thm]{Proposition}
\newtheorem{cor}[thm]{Corollary}
\newtheorem{claim}[thm]{Claim}

\theoremstyle{definition}
\newtheorem{defn}[thm]{Definition}

\newtheorem{exmp}[thm]{Example}
\newtheorem{rem}[thm]{Remark}

\theoremstyle{remark}

\title{Decomposition of exact pfd persistence bimodules}
\author[J\'er\'emy Cochoy]{Jérémy Cochoy}
\email{jeremy.cochoy@inria.fr}
\author[Steve Oudot]{Steve Oudot}
\email{steve.oudot@inria.fr}
\address{Inria Saclay -- \^Ile-de-France}

\begin{document}

\begin{abstract}
We characterize the class of persistence modules indexed over $\sR^2$
that are decomposable into summands whose support have the shape of a
{\em block}---i.e. a horizontal band, a vertical band, an upper-right
quadrant, or a lower-left quadrant. Assuming the modules are pointwise finite dimensional (pfd), we show that they are
decomposable into block summands if and only if they satisfy a certain
local property called {\em exactness}. Our proof follows the same
scheme as the proof of decomposition for pfd persistence modules
indexed over $\sR$, yet it departs from it at key stages due to the
product order on $\sR^2$ not being a total order, which leaves some
important gaps open. These gaps are filled in using more direct
arguments. Our work is motivated primarily by the stability theory for
zigzags and interlevel-sets persistence modules, in which
block-decomposable bimodules play a key part. Our  results
allow us to drop some of the conditions under which that
theory holds, in particular the Morse-type conditions. 
\end{abstract}

\maketitle

\section{Introduction}

Decomposition theorems are one of the pillars of topological
persistence theory, as they provide sufficient conditions under which
topological descriptors for data, called {\em barcodes} or {\em
  diagrams}, can be defined. We refer the reader
to~\cite{o-ptfqrtda-15} for an introduction to persistence theory and
to the role played by decompositions therein.

The 1-dimensional setting is pretty well-understood by now. In this setting,
the objects of interest are the so-called {\em 1-d persistence
  modules}, which are functors from the real line~$\sR$ (equipped with
its natural order) to the vector spaces over a fixed field. The
category of such functors is abelian, and several
theorems~\cite{a-rtaa-74,botnan2015interval,cb,Gabriel72Quiver,rt-qf3-75,w-dgm-85}
identify conditions under which these functors decompose as direct
sums of {\em interval modules}---functors that are constantly equal to
the field over some interval and trivial elsewhere. Hence the notion of {\em barcode} of a
persistence module, defined as the collection of such intervals
appearing in its (essentially unique) decomposition.

By contrast, the higher-dimensional setting is much less
understood. While decomposition theorems
exist~\cite{botnan2018decomposition,r-izmir-6}, the underlying quivers
are {\em wild-representation type}, meaning that their sets of
indecomposables are hard to classify~\cite{cz-mp-09}. Hence the
difficulties to extend the concept of barcode to higher dimensions.

In this paper we consider the 2-dimensional setting and we
characterize the subclass of the $\sR^2$-indexed modules that
decompose into {\em block modules}, i.e. functors that are constantly
equal to the field on 2-d shapes called {\em blocks} (upper-right or
lower-left quadrants, vertical or horizontal bands) and trivial
elsewhere. We prove that these modules are precisely the ones that
satisfy a certain local property called {\em exactness} 
(Theorem~\ref{thm_key}). Thus, we provide a necessary and sufficient local
condition (exactness) by which a global property (decomposition into
block modules) can be checked.  We show some applications of this
result, most notably to the study of {\em zigzag persistence} and the
development of a stability theory for {\em interlevel-sets persistence}
based on earlier work by Bjerkevik, Botnan and
Lesnick~\cite{b-shdidpm-16,bl-aszpm-16}.

\subsubsection*{Related work.}

To our knowledge, our result was historically the first decomposition theorem
proposed for $\sR^2$-indexed persistence modules (or a certain class
thereof, not restricted to grid-indexed modules).  Since its first
appearance, a general Krull-Schmidt theorem for pointwise
finite-dimensional modules indexed over small categories has been proven by
M. Botnan and W. Crawley-Boevey~\cite{botnan2018decomposition}
following an approach developed by C. Ringel for locally finite
quivers~\cite{r-izmir-6}, from which another proof of our result can
be derived by studying the structure of the indecomposables.

Let us also mention a different and complementary approach to the problem,
which consists in defining barcodes via rectangle measures in the
plane. This definition does not require any decomposition, thus it
allows to relax somewhat the hypotheses on the considered persistence
module.  The approach was initiated in~\cite{chazal2012structure} in
the 1-d setting, then it was extended to exact 2-d modules
in~\cite{dan2017new,cdsm-phzp-16}. As it does not yield any
decomposition theorem, it cannot be combined with the stability theory
developed in~\cite{b-shdidpm-16,bl-aszpm-16}, however it constitutes a
serious alternative to our work.

The special case of interlevel-sets persistence yields exact 2-d modules with  some specific properties. Most notably,
the structure of such modules is fully  determined by the interval decomposition of their restriction to some zigzag along the anti-diagonal---by the Mayer-Vietoris theorem, assuming all homology degrees are considered at once. This fact was leveraged  in early work on the subject~\cite{bendich2013homology,carlsson2009zigzag}, which established a decomposition theorem for such modules in the discrete setting\footnote{By {\em discrete setting} we mean that the indexing set sits inside the integer lattice~$\sZ^2$. In~\cite{bendich2013homology,carlsson2009zigzag}  the result is established in the case where $\sZ^2$ is replaced with a finite grid, however the analysis extends to the entire lattice via the more recent interval decomposition theorem by Botnan~\cite{botnan2015interval} for $\sZ$-indexed zigzag modules.} via the interval decomposition of zigzag modules.
These early contributions were seminal in that they introduced the problem, proved a discrete analogue of our Theorem~\ref{thm_key} in the special case of interlevel-sets persistence, and provided some of the  key ideas that were exploited in later work such as~\cite{cdsm-phzp-16}.

\subsection*{Acknowledgements}

Theorem~\ref{thm_key} was introduced to us as a conjecture by
M. Botnan and M. Lesnick. We would like to thank them for leading us
to this question in the first place.  We would also like to thank the
reviewers for their insightful comments and suggestions, which greatly
helped shape the paper and simplify its exposition.  This work was
supported in part by ERC grant Gudhi (ERC-2013-ADG-339025) and by ANR
project TopData (ANR-13-BS01-0008).

\section{Main result}

Throughout the exposition, a field of coefficients is fixed and
denoted by $\field$. The set on which the vector spaces of our modules will be indexed is~$\sR^2$, equipped with the product order:
\[
\forall s,t\in\sR^2,\quad s\leq t \quad \Longleftrightarrow \quad s_x\leq t_x\ \mbox{and}\ s_y\leq t_y.
\]
A {\em persistence module} indexed over $\sR^2$ (or {\em persistence bimodule}\footnote{We abuse terminology here,  bimodules being a different concept in abstract algebra.}
for short) is a functor~$M$ from the poset $(\sR^2, \leq)$ to the category of vector spaces over~$\field$. By default we will
denote by $\M_t$, $t\in\sR^2$, its constituent vector spaces, and by $\rho_s^t$, $s\leq t\in\sR^2$, its constituent linear maps. For clarity, $\rho_s^t$ will be sometimes renamed $v_s^t$ when $s_x=t_x$ (`$v$' for `vertical'), and $h_s^t$ when $s_y=t_y$ (`$h$' for `horizontal').
For any $s\leq t\in\sR^2$ we have the following commutative diagram where 
the spaces and maps are taken from~$M$:
\begin{equation}\label{eq:quadrangle}
\begin{gathered}
\xymatrix{
\M_{(s_x, t_y)}\ar^-{h_{(s_x, t_y)}^t}[rr] && \M_t\\\\
\M_s\ar_-{h_s^{(t_x, s_y)}}[rr]\ar^-{v_s^{(s_x, t_y)}}[uu] && \M_{(t_x, s_y)}\ar_-{v_{(t_x, s_y)}^t}[uu]
}
\end{gathered}
\end{equation}
$M$ is called {\em pointwise finite-dimensional} (pfd) if $\M_t$ is
finite-dimensional for every $t\in\sR^2$. It is called {\em exact} if,
for every $s\leq t\in\sR^2$, the following sequence induced by diagram~(\ref{eq:quadrangle}) is
exact (i.e. $\Ima \phi = \Ker \psi$):
\[\xymatrix@C=110pt{
\M_s \ar^-{\scriptstyle \phi\,=\,\left(h_s^{(t_x, s_y)},\; v_s^{(s_x, t_y)}\right)}[r] & \M_{(t_x, s_y)} \oplus \M_{(s_x, t_y)} \ar^-{\psi\,=\,v_{(t_x, s_y)}^t - h_{(s_x, t_y)}^t}[r] & \M_t.
}\]
There are several ways in which this condition can be interpreted, including:
\begin{itemize}
\item At a low level, $\Ima\phi\subseteq\Ker\psi$ means that diagram~(\ref{eq:quadrangle}) commutes, while $\Ima\phi\supseteq\Ker\psi$ means that every element of $M_t$ that has preimages in $M_{(t_x,s_y)}$ and $M_{(s_x,t_y)}$ has a preimage common to both in~$M_s$.
\item At a higher level, exactness of the sequence means that diagram~(\ref{eq:quadrangle}) is a weak form of pushout or pullback square, in which surjectivity of~$\psi$  or injectivity of~$\phi$ are not required.
\end{itemize}
The exactness condition also implies (and is stronger than) the following equalities, which will be instrumental in our analysis:
\begin{equation}\label{eq:weak_exactness}
  \begin{gathered}
\begin{array}{rcl}
  \Ker \rho_s^t &=& \Ker h^{(t_x,s_y)}_s + \Ker v^{(s_x, t_y)}_s  \\[0.5em]
  \Ima \rho_s^t &=& \Ima v_{(t_x,s_y)}^t  \cap \Ima h_{(s_x, t_y)}^t
  \end{array}
  \end{gathered}
  \end{equation}

%
In this paper we are interested in
exact pfd bimodules. In some places our analysis extends to modules satisfying~(\ref{eq:weak_exactness}), for which the natural shapes to consider are rectangles.

\subsection*{Rectangles, blocks, and their associated modules}

We use {\em
  cuts} to parametrize the rectangles in the plane.
A cut is a partition $\cut$ of
$\sR$ into two (possibly empty) sets $\cut^-, \cut^+$ such that $x<y$
for all $x\in \cut^-$ and $y\in \cut^+$. For instance, $\cut =(\cut^-,
\cut^+)$ with $\cut^- = (-\infty, 1]$ and $\cut^+ = (1, +\infty)$ is a
  cut. A cut $\cut$ with either $\cut^+=\emptyset$ or
  $\cut^-=\emptyset$ is said to be at infinity or trivial.

A non-empty rectangle~$R$ in the plane is then uniquely defined by
four cuts: two horizontal (say $\lcut$ and $\rcut$, standing respectively for {\em left cut} and {\em right cut}), and two vertical
(say $\bcut$ and $\tcut$, standing respectively for {\em bottom cut} and {\em top cut}), so that $R = (\lcut^+\cap
\rcut^-) \times (\bcut^+ \cap
\tcut^-)$. Note that $R$ may not necessarily be open or closed, in
fact the nature of each cut determines which boundaries belong to the
rectangle.

To any rectangle~$\Rs$ we associate a unique {\em rectangle module}
$\field_{\Rs}$ having a copy of the field $\field$ at every point
$t\in\Rs$ and zero vector spaces elsewhere, the copies of $\field$
being connected by identities and the rest of the maps being zero. It
is easily seen that any such bimodule is pfd and satisfies the
equalities of~(\ref{eq:weak_exactness}). However, not every rectangle module is exact, as illustrated in Figure~\ref{fig:rect_not-exact}.

\begin{figure}[htb]
  \centering
  \includegraphics[scale=0.7]{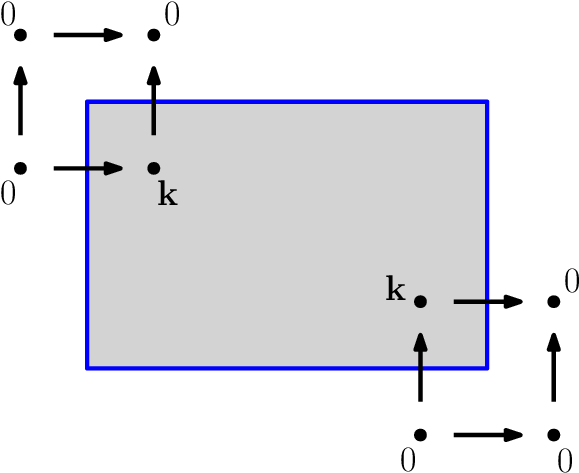}
  \caption{A rectangle (in gray with blue boundary) and its corresponding rectangle module, which is not exact around the upper-left and bottom-right corners---the highlighted squares have $\Ima\phi=0\not\simeq \field\simeq\Ker\psi$.}
  \label{fig:rect_not-exact}
\end{figure}

Among the rectangles with at least two cuts at infinity, we
distinguish the following four types, illustrated in
Figure~\ref{fig:blocks}:
\begin{itemize}
\item {\em birth quadrants} (shorthand: bquad), for which $\rcut^+=\tcut^+=\emptyset$;
\item {\em death quadrants} (shorthand: dquad), for which $\lcut^-=\bcut^-=\emptyset$;
\item {\em horizontal bands} (shorthand: hband), for which $\lcut^-=\rcut^+=\emptyset$;
\item {\em vertical bands} (shorthand: vband), for which $\bcut^-=\tcut^+=\emptyset$.
\end{itemize}
\begin{figure}[htb]
  \centering
  \includegraphics[width=\textwidth]{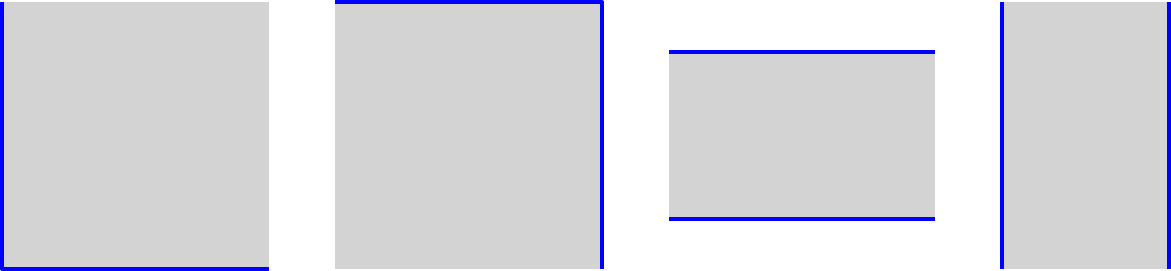}
  \caption{The four block types. From left to right: birth quadrant (bquad), death quadrant (dquad), horizontal band (hband), vertical band (vband).}
  \label{fig:blocks}
\end{figure}
A {\em block} is a rectangle of any of these types. Note that these
types are not mutually exclusive, for instance $\sR^2$ belongs to all
four of them. The rectangle module associated to a given block is
called a block module.  It is immediate that any such bimodule is both
pfd and exact, and that, among the rectangles in the plane, only those
that are blocks yield exact modules---see the counterexample in
Figure~\ref{fig:rect_not-exact}. Since being exact is invariant under
taking direct sums, any pfd bimodule that is decomposable into block
summands (or {\em block-decomposable} for short) is also exact.  Our
main result states that the converse is also true:
\begin{thm}[Decomposition of exact pfd bimodules]\label{thm_key}
Any exact pfd  bimodule $M$ decomposes as a direct sum of block modules: 
\begin{equation}\label{eq:decomp}
M \simeq \bigoplus\limits_{\Fb\in\barcode(M)} \field_{\Fb},
\end{equation}
where $\barcode(M)$ is some multiset of
blocks that depends on $M$. The decomposition is unique up to
isomorphism and reordering of the terms.
\end{thm}
Thus, among the pfd bimodules, the ones that are block-decomposable
are precisely the ones that are exact. Some applications of this
result are described in Section~\ref{sec:applications}. Among them,
the study of the stability of {\em zigzags} in the context of {\em
  interlevel-sets persistence} (Section~\ref{sec:interlevel-sets})
served as the initial motivation for this work. Exactness in that
setting is ensured by the Mayer-Vietoris theorem.

\subsection*{Proof outline}

The bulk of the paper (Sections~\ref{sec:cuts} to \ref{sec:completion})
is devoted to the proof of Theorem~\ref{thm_key}. The uniqueness of
the decomposition is a straightforward consequence of Azumaya's
theorem~\cite{a-kst-50}, the endomorphism ring of any block module
being clearly isomorphic to $\field$ and therefore local.
It remains to prove the existence of a decomposition, for which we
follow the same scheme as in the 1-d setting~\cite{cb}, with some significant adjustments at each step due to the product order on~$\sR^2$ not being total:
\begin{itemize}
\item In Sections~\ref{sec:cuts} and~\ref{sec:counting functor}, to
  each rectangle~$\Rs$ in the plane we associate a {\em counting
    functor} $C_\Rs$ (see~(\ref{eq:CR})) that maps pfd persistence
  bimodules satisfying~(\ref{eq:weak_exactness}) to $\field$-vector
  spaces. This functor captures the elements whose lifespan is
  exactly~$\Rs$ in such a module~$M$. In particular, if we assume~$M$
  to be decomposable into rectangle summands, then the dimension of
  the space~$C_\Rs(M)$ is the same as the multiplicity of the
  summand~$\field_\Rs$ in the decomposition of~$M$
  (Lemma~\ref{lem:C_f_mult}).  The functor is built using the so-called {\em
    functorial filtration} technique
  (see e.g.~\cite{ringel1975indecomposable}), which consists in filtering each
  space~$M_t$ by the kernels and images of the internal morphisms
  $\rho_s^t$ (for $s\leq t$) and $\rho_t^u$ (for $u\geq t$).  The
  technique was used in the 1-d setting, where the total order on the
  real line made it simple to show that the image and kernel subspaces
  get transported from one index to the other within $\Rs$, which is
  the key property to define the counting functor. Here the
  transportation property is preserved
  (Corollary~\ref{cor:transportation}) modulo some adjustments to the
  definitions of the image and kernel subspaces, leveraging the
  equalities in~(\ref{eq:weak_exactness}). See~(\ref{eq:fs_ima_ker})
  for the precise definitions, Remark~\ref{rem:IK} for the underlying
  intuition, and Example~\ref{ex:IK} for an illustrative example.
\item From Section~\ref{sec:submodules} onward, we assume that the
  bimodule~$M$ is pfd and exact. In order to build an explicit
  internal decomposition of~$M$, to any block~$\Fs$ we associate a 
  submodule $\MRs_\Fs$ of~$M$ whose structure is that of a direct sum
  of $\dim C_\Fs(M)$ copies of~$\field_\Fs$ (see Lemma~\ref{lem:block-copies}). As in the 1-d setting, the construction of~$\MRs_\Fs$ leverages that of the counting functor~$C_\Fs$, and it boils down to choosing some vector-space complement in a certain inverse limit. However, in contrast to the 1-d case, the vector-space complement cannot be chosen arbitrarily, as care must be taken of the way the image of the complement through the cone maps transitions outside the block~$\Fs$ (see Proposition~\ref{prop:MRs}). 
\item In Section~\ref{sec:direct_sum} we show that the submodules
  $\MRs_{\Fs}$ are in direct sum, which it is sufficient to check
  pointwise at every index $t\in\sR^2$. The proof in the 1-d setting
  uses the concept of {\em disjoint sections} of a vector space, from
  which the direct sum follows. In our setting, while the submodules
  associated to bands satisfy the disjointness property leveraged in
  1-d, the whole family of
  submodules~$\{\MRs_\Fs\}_{\Fs:\mathrm{block}}$ does not (see
  Example~\ref{ex:disjoint}). We therefore resort to more direct
  arguments, based on the observation that modules associated to
  different blocks have different supports, so that showing the direct
  sum amounts for the most part to showing that, for a finite family
  of linearly related block modules, there is one whose support
  extends further than the others to the right or to the top
  (Propositions~\ref{prop:direct_sum_in}
  and~\ref{prop:direct_sum_inter}).  This fact is not true in all
  cases, but sufficienty widely so that the special cases
  can be handled individually using exactness.
\item In Sections~\ref{sec:sections} and~\ref{sec:completion} we show
  that the submodules $\MRs_{\Fs}$ generate the whole module~$M$,
  which it is also sufficient to check pointwise at every index
  $t\in\sR^2$.  The proof in the 1-d setting uses the concept of {\em
    covering sections}, from which the result follows. In our setting,
  the family of submodules~$\{\MRs_\Fs\}_{\Fs:\mathrm{block}}$ does
  not satisfy the covering property (see Example~\ref{ex:cover}),
  unless we remove from~$M$ those elements that live since $-\infty$
  both horizontally and vertically (Proposition~\ref{bigcover}). We
  therefore study separately the submodule generated by those
  elements, showing by a duality argument that it itself decomposes as
  a direct sum of block modules (Corollary~\ref{cor:N_decomp}).  This
  allows us to complete the construction of the internal direct-sum
  decomposition of~$M$ (Corollary~\ref{cor:internal-ds}).
  %
\end{itemize}
%

%
%

\section{Images and kernels}
\label{sec:cuts}

As in the one-dimensional setting~\cite{cb}, the basic
ingredients in our analysis are certain limits
of images and kernels. Given a rectangle~$R=(\lcut^+\cap\rcut^-)\times(\bcut^+\cap\tcut^-)$, and a point $t\in R$,
we first construct these limits along the 1-dimensional restrictions of the module~$M$ to the horizontal
and vertical lines passing through~t (with the convention that $\Ima_{\cut, t}^- = 0$ when $\cut^-=\emptyset$ and $\Ker_{\cut, t} = M_t$ when $\cut^+=\emptyset$):
\begin{equation}\label{eq:ima_ker_1d}
  \begin{gathered}
    \begin{array}{lcl}
    \Ima_{\lcut, t}^+(M) = \bigcap\limits_{\begin{smallmatrix}x \in \lcut^+\\x \leq t_x\end{smallmatrix}} \Ima h_{(x, t_y)}^t
  &\quad&
  \Ima_{\lcut, t}^-(M) = \sum\limits_{x \in \lcut^-} \Ima h_{(x, t_y)}^t\\[2em]
  \Ker_{\rcut, t}^+(M) = \bigcap\limits_{x \in \rcut^+} \Ker h_t^{(x, t_y)}
&&
  \Ker_{\rcut, t}^-(M) = \sum\limits_{\begin{smallmatrix}x \in \rcut^-\\x\geq t_x\end{smallmatrix}} \Ker h_t^{(x, t_y)}\\[2em]
  \Ima_{\bcut, t}^+(M) = \bigcap\limits_{\begin{smallmatrix}y \in \bcut^+\\y \leq t_y\end{smallmatrix}} \Ima h_{(t_x, y)}^t
  &\quad&
  \Ima_{\bcut, t}^-(M) = \sum\limits_{y \in \bcut^-} \Ima h_{(t_x, y)}^t\\[2em]
  \Ker_{\tcut, t}^+(M) = \bigcap\limits_{y \in \tcut^+} \Ker h_t^{(t_x, y)}
&&
  \Ker_{\tcut, t}^-(M) = \sum\limits_{\begin{smallmatrix}y \in \tcut^-\\y\geq t_y\end{smallmatrix}} \Ker h_t^{(t_x, y)}
 \end{array} \end{gathered}
\end{equation}
See Figure~\ref{fig:ImaKer} for an illustration. In the following we omit $M$ from our notations whenever the considered module is obvious.
\begin{figure}[htb]
\centering
\includegraphics[scale=0.8]{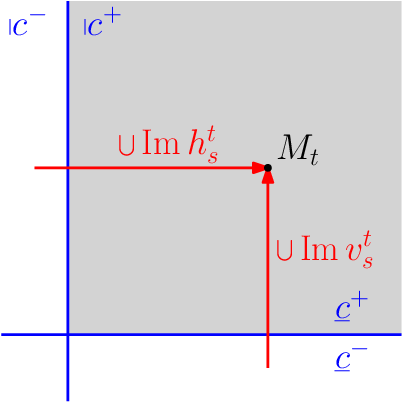}\hfill
\includegraphics[scale=0.8]{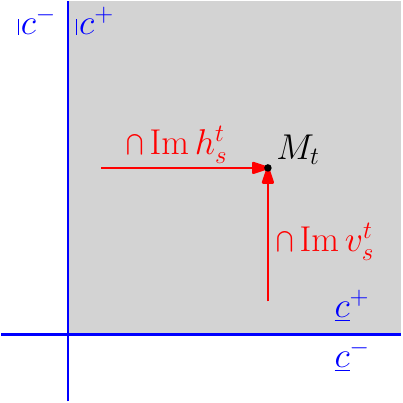}\\[5ex]
\includegraphics[scale=0.8]{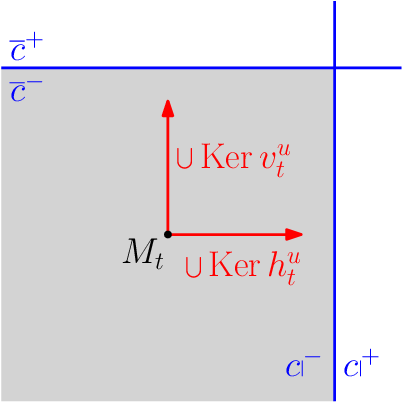}\hfill
\includegraphics[scale=0.8]{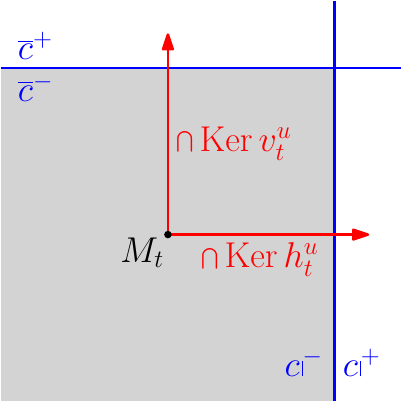}
\caption{From top to bottom and from left to right: the spaces $\Ima^-_{\cut,t}$, $\Ima^+_{\cut,t}$, $\Ker^-_{\cut,t}$ and $\Ker^+_{\cut,t}$.}
\label{fig:ImaKer}
\end{figure}
\begin{lem}[Realization]\label{lem:realization}
Assume $M$ is pfd, and extend it to a representation of the extended plane
$[-\infty, +\infty]^2$ by letting $M_{(\pm\infty,\cdot)} = M_{(\cdot,
  \pm\infty)} = 0$. Then:
\begin{align*}
\Ima^+_{\lcut, t} &= \Ima h_{(x,t_y)}^t \ \mbox{for some}\ x\in \lcut^+\cap (-\infty,t_x]\ \mbox{and any lower}\ x\in \lcut^+\\ 
\Ima^-_{\lcut, t} &= \Ima h_{(x,t_y)}^t \ \mbox{for some}\ x\in \lcut^-\cup\{-\infty\}\ \mbox{and any greater}\ x\in \lcut^-\\ 
\Ker^+_{\rcut, t} &= \Ker h_t^{(x,t_y)} \ \mbox{for some}\ x\in \rcut^+\cup\{+\infty\}\ \mbox{and any lower}\ x\in \rcut^+\\ 
\Ker^-_{\rcut, t} &= \Ker h_t^{(x,t_y)} \ \mbox{for some}\ x\in \rcut^-\cap [t, +\infty)\ \mbox{and any greater}\ x\in \rcut^- 
\end{align*}
And similarly for the vertical cuts $\bcut, \tcut$. Note that the
spaces $\Ima^\pm_{\cut,t}$ and $\Ker^\pm_{\cut,t}$ mentioned here, which are those of the extension of~$M$, are the same as those of~$M$ since $t\in\sR^2$
  and the cuts considered are cuts of $\sR$.
\end{lem}
\begin{proof}
This is Lemma~2.1 from~\cite{cb}, and a direct consequence of the finite dimensionality of~$M_t$.
\end{proof}
We now combine the contributions from the horizontal and vertical restrictions of~$M$ at point~$t$ as follows (where equalities between formulas come from the inclusions $\Ima_{\cut,t}^-\subseteq \Ima_{\cut,t}^+$ and $\Ker_{\cut,t}^-\subseteq \Ker_{\cut,t}^+$):
\begin{equation}\label{eq:fs_ima_ker}
\begin{gathered}
\begin{array}{rcl}
\Ima_{\Rs, t}^+ &=&  \Ima^+_{\lcut, t} \cap \Ima_{\bcut, t}^+,\\[0.5em]
\Ima_{\Rs, t}^- &=&  (\Ima^-_{\lcut, t} + \Ima_{\bcut, t}^-) \cap \Ima_{\Rs, t}^+\\[0.5em]
&=& \Ima^-_{\lcut, t} \cap \Ima^+_{\bcut,t} + \Ima_{\bcut, t}^-\cap \Ima^+_{\lcut,t}\\[0.5em]
\Ker_{\Rs, t}^+ &=&  (\Ker^+_{\rcut, t} + \Ker^-_{\tcut, t}) \cap (\Ker^-_{\rcut, t} + \Ker^+_{\tcut, t})\\[0.5em]
&=&  \Ker^+_{\rcut, t}\cap \Ker^+_{\tcut, t}   + \Ker^-_{\rcut, t} + \Ker^-_{\tcut, t}\\[0.5em]
\Ker_{\Rs, t}^- &=& \Ker^-_{\rcut, t} + \Ker^-_{\tcut, t}
\end{array}
\end{gathered}
\end{equation}
It is immediate from the definition that $\Ima^-_{\Rs,t}\subseteq \Ima^+_{\Rs,t}$ and $\Ker^-_{\Rs,t}\subseteq \Ker^+_{\Rs,t}$.

\begin{rem}\label{rem:IK}
  The above definitions are motivated as follows. The straightforward
  generalization of the definitions from the 1-d setting~\cite{cb}
  gives the following spaces:
%
  \begin{equation}\label{eq:I+-K+-}
  \begin{gathered}
  \begin{array}{lcl}
    I^+_{\Rs,t} = \bigcap\limits_{\begin{smallmatrix}s\in\Rs\\s\leq t\end{smallmatrix}} \Ima\rho_s^t   &\quad&
    I^-_{\Rs,t} = \sum\limits_{\begin{smallmatrix}s\notin\Rs\\s\leq t\end{smallmatrix}} \Ima\rho_s^t \\[2em]
    K^+_{\Rs,t} = \bigcap\limits_{\begin{smallmatrix}u\notin\Rs\\u\geq t\end{smallmatrix}} \Ker\rho_t^u   &\quad&
    K^-_{\Rs,t} = \sum\limits_{\begin{smallmatrix}u\in\Rs\\u\geq t\end{smallmatrix}} \Ker\rho_t^u
  \end{array}
  \end{gathered}
\end{equation}
Unfortunately, due to the order~$\leq$ not being total on $\sR^2$, we may not  always have $I^-_{\Rs,t} \subseteq I^+_{\Rs,t}$ nor $K^-_{\Rs,t}\subseteq K^+_{\Rs,t}$. The fix is to consider sums and intersections as follows:
\begin{equation}\label{eq:fs_ima_ker_bis}
  \begin{gathered}
    \begin{array}{rcl}
\Ima^+_{\Rs,t} = I^+_{\Rs,t}  &\quad&
\Ima^-_{\Rs,t} = I^-_{\Rs,t}\cap I^+_{\Rs,t}\\[1em]
\Ker^+_{\Rs,t} = K^+_{\Rs,t} + K^-_{\Rs,t} &\quad&
\Ker^-_{\Rs,t} = K^-_{\Rs,t}
    \end{array}
  \end{gathered}
\end{equation}
Other combinations of sums and intersections could be considered,
e.g. letting $\Ima^+_{\Rs,t} = K^+_{\Rs,t}+K^-_{\Rs,t}$ and
$\Ima^-_{\Rs,t} = I^-_{\Rs,t}$, however the ones above are the only ones
ensuring that the spaces can be transported from one index~$t$ to
another $t'\geq t$ (see the Transportation
Corollary~\ref{cor:transportation} below). Furthermore, they induce a
 duality between image and kernel spaces, through vector-space
duality, as will be emphasized and exploited in
Section~\ref{sec:completion} (see Lemma~\ref{lem:bot}). Finally,
assuming the module~$M$ is exact or satisfies the equalities
of~(\ref{eq:weak_exactness}), the horizontal and vertical
contributions to the images and kernels in~(\ref{eq:I+-K+-}) can be
decoupled so that the definitions in~(\ref{eq:fs_ima_ker_bis}) are
equivalent to those in~(\ref{eq:fs_ima_ker}) (see
Appendix~\ref{sec:proof-equiv}). In the following proof of
Theorem~\ref{thm_key} we only use the definitions
from~(\ref{eq:fs_ima_ker}), not the ones from~(\ref{eq:fs_ima_ker_bis}),
for consistency.
\end{rem}

\begin{exmp}\label{ex:IK}
Take for~$M$ the
direct sum of the modules associated with a birth quadrant~$A$ and a
death quadrant~$B$, such that the intersection~$A\cap B$ is
non-empty (see Figure~\ref{fig:BD} for an illustration).
    \begin{figure}[htb]
  \centering
  \includegraphics[scale=0.5]{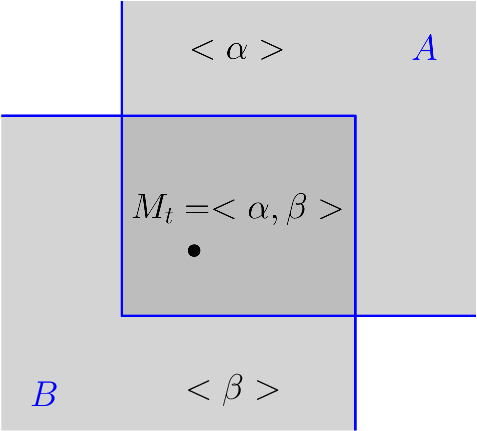}
  \caption{Overlapping birth and death quadrants.}
  \label{fig:BD} 
  \end{figure}
Given any $t\in A\cap B$, call $\alpha$ a
generator of the 1-dimensional subspace of~$M_t$ spanned by~$A$, and
$\beta$ a counterpart for~$B$. Then, applying the formulas
in~(\ref{eq:fs_ima_ker}) (or equivalently
the ones in~(\ref{eq:fs_ima_ker_bis})), we get:
\[
\begin{array}{llll}
  \Ima^+_{A,t} = M_t & \Ima^-_{A,t} = \langle\beta\rangle & \Ker^+_{A,t} = M_t & \Ker^-_{A,t} = \langle\beta\rangle \\[0.5em]
  \Ima^+_{B,t} = \langle\beta\rangle & \Ima^-_{B,t} = 0 & \Ker^+_{B,t} = \langle\beta\rangle & \Ker^-_{B,t} = 0
\end{array}
\]
Then, we see that:
\[
\begin{array}{ll}
  \Ima^+_{A,t}/\Ima^-_{A,t} \simeq \langle\alpha\rangle \simeq \Ker^+_{A,t}/\Ker^-_{A,t}\\[0.5em]
  \Ima^+_{B,t}/\Ima^-_{B,t} \simeq \langle\beta\rangle \simeq \Ker^+_{B,t}/\Ker^-_{B,t}
\end{array}
\]
Thus, for each block $A,B$ the quotient $\Ima^+/\Ima^-$ captures
those elements that are born on the bottom-left boundary of the block,
while $\Ker^+/\Ker^-$ captures those that die on its top-right
boundary. This fact holds generally for blocks, and (beyond that) also
for rectangles, as we shall see in Section~\ref{sec:counting
  functor}. Moreover, it is independent of the choice of index~$t$ within
the block or rectangle, as a consequence of the following {\em
  transportation} property.
\end{exmp}
 
It turns out that images get transported to images, and kernels to
kernels, through the internal morphisms of~$M$---see
Corollary~\ref{cor:transportation} below. The proof says something
slightly more precise, namely:
\begin{lem}\label{lem:trans}
Assume $M$ is pfd and satisfies~(\ref{eq:weak_exactness}). Let
$\Rs=(\lcut^+\cap\rcut^-)\times(\bcut^+\cap\tcut^-)$ be a rectangle,
$s\leq t\in \Rs$, and $\bullet, \blacktriangle \in \{+,-\}$.  Then:
\begin{align*}
  \rho_s^{t}(\Ima_{\lcut, s}^\bullet\cap \Ima_{\bcut,s}^\blacktriangle) &= \Ima_{\lcut, t}^\bullet\cap \Ima_{\bcut, t}^\blacktriangle\\
  (\rho_s^{t})^{-1}(\Ker_{\rcut, t}^\bullet + \Ker_{\tcut, t}^\blacktriangle) &= \Ker_{\rcut, s}^\bullet + \Ker_{\tcut, s}^\blacktriangle
\end{align*}
\end{lem}
\begin{proof}
  For convenience, we extend $M$ to a representation of the extended
  plane $[-\infty, +\infty]^2$ by letting $M_{(\pm\infty,\cdot)} =
  M_{(\cdot,\pm\infty)} = 0$. Note that this extension may not be
  exact when~$M$ itself is exact, however it still satisfies the
  equalities of~(\ref{eq:weak_exactness}) when $M$ does so.
  
  \medskip

  We first consider images.  The Realization
  Lemma~\ref{lem:realization} tells us that there exist $x\leq s_x\leq
  t_x$ and $y\leq s_y\leq t_y$ (possibly equal to $-\infty$) such that
  \begin{align*}
    \Ima_{\lcut, s}^\bullet = \Ima h_{(x,s_y)}^s \ \mbox{and}\ \Ima_{\lcut, t}^\bullet = \Ima h_{(x,t_y)}^t\\
    \Ima_{\bcut, s}^\blacktriangle = \Ima v_{(s_x,y)}^s \ \mbox{and}\ \Ima_{\bcut, t}^\blacktriangle = \Ima v_{(t_x,y)}^t
  \end{align*}
  We then have the following commutative diagram:
  \[ \xymatrix{
    (x, t_y) \ar[rr] && t\\
    (x, s_y) \ar[r] \ar[u] & s \ar[ur]\\
    (x,y) \ar[u] \ar[ur] \ar[r] & (s_x, y) \ar[u] \ar[r] & (t_x, y) \ar[uu]
  }\]
  Chasing through this diagram gives:
  \begin{align*}
    \Ima h_{(x, t_y)}^t \cap \Ima v_{(t_x, y)}^t &\stackrel{\mbox{\scriptsize (Eq. \ref{eq:weak_exactness})}}{=} \Ima \rho_{(x,y)}^t
    = \rho_s^t (\Ima \rho_{(x,y)}^s)\\
    &\stackrel{\mbox{\scriptsize (Eq. \ref{eq:weak_exactness})}}{=} \rho_s^t (\Ima h_{(x,s_y)}^s \cap \Ima v_{(s_x, y)}^s)
    \end{align*}
    thus proving the first part of the lemma.

    \medskip

  We now consider kernels.  The Realization
  Lemma~\ref{lem:realization} tells us that there exist $x\geq t_x\geq
  s_x$ and $y\geq t_y\geq s_y$ (possibly equal to $+\infty$) such that
  \begin{align*}
    \Ker_{\rcut, s}^\bullet = \Ker h^{(x,s_y)}_s \ \mbox{and}\ \Ker_{\rcut, t}^\bullet = \Ker h^{(x,t_y)}_t\\
    \Ker_{\tcut, s}^\blacktriangle = \Ker v^{(s_x,y)}_s \ \mbox{and}\ \Ker_{\tcut, t}^\blacktriangle = \Ker v^{(t_x,y)}_t
  \end{align*}
  We then have the following commutative diagram:
  \[ \xymatrix{
    (s_x, y) \ar[r] & (t_x, y) \ar[r] & (x,y)\\
    & t \ar[r] \ar[u] \ar[ur] & (x, t_y) \ar[u]\\
    s \ar[uu] \ar[ur] \ar[rr] & & (x, s_y) \ar[u]
  }\]
  Chasing through this diagram gives:
  \begin{align*}
  (\rho_s^{t})^{-1}(\Ker h^{(x, t_y)}_t + \Ker v^{(t_x, y)}_t)
  &\stackrel{\mbox{\scriptsize (Eq. \ref{eq:weak_exactness})}}{=}
  (\rho_s^{t})^{-1}(\Ker \rho^{(x, y)}_t)
  = \Ker \rho^{(x, y)}_s \\  
   & \stackrel{\mbox{\scriptsize (Eq. \ref{eq:weak_exactness})}}{=} \Ker h^{(x, s_y)}_s + \Ker v^{(s_x, y)}_s
  \end{align*}
    thus proving the second part of the lemma.
\end{proof}

\begin{cor}[Transportation]\label{cor:transportation}
Assume $M$ is pfd and satisfies~(\ref{eq:weak_exactness}). Let $\Rs$
be a rectangle and let $s\leq t\in \Rs$.  Then (using $\pm$ as a
shorthand for either $+$ or $-$, with the same sign on both sides of
the equality):
\begin{align*}
  \rho_s^{t}(\Ima_{\Rs, s}^\pm) &= \Ima_{\Rs, t}^\pm\\
  (\rho_s^{t})^{-1}(\Ker_{\Rs, t}^\pm) &= \Ker_{\Rs, s}^\pm
\end{align*}
\end{cor}
\begin{proof}
Follows from Lemma~\ref{lem:trans} and from the facts that $f(U+V) = f(U)+f(V)$ and $f^{-1}(U\cap V) = f^{-1}(U)\cap f^{-1}(V)$---recall the definitions in~(\ref{eq:fs_ima_ker}).
\end{proof}

Another property that we will be using is that, whenever the module~$M$ is exact and $\Fs$ is a block, kernels are included in images as follows:
\begin{lem}\label{b-ki}
Assume $M$ is pfd and exact. Then, for any fixed $t\in\sR^2$ and block $\Fs = (\lcut^+\cap\rcut^-)\times (\bcut^+\cap \tcut^-)$ containing~$t$:
\begin{itemize}
\item $\Ker_{\rcut,t}^- \subseteq \Ima_{\bcut,t}^+$ and
  $\Ker_{\tcut,t}^- \subseteq \Ima_{\lcut,t}^+$.
\item If $\rcut^+ \neq\emptyset$ (resp. $\tcut^+ \neq\emptyset$ ) then $\Ker_{\rcut,t}^+ \subseteq \Ima_{\bcut,t}^+$ (resp. $\Ker_{\tcut,t}^+ \subseteq \Ima_{\lcut,t}^+$).
\item If $\bcut^- \neq\emptyset$ (resp. $\lcut^- \neq\emptyset$ ) then $\Ker_{\rcut,t}^- \subseteq \Ima_{\bcut,t}^-$ (resp. $\Ker_{\tcut,t}^- \subseteq \Ima_{\lcut,t}^-$).
\item If both $\rcut^+ \neq\emptyset \neq \bcut^-$ (resp. $\tcut^+ \neq\emptyset \neq \lcut^-$) then $\Ker_{\rcut,t}^+ \subseteq \Ima_{\bcut,t}^-$ (resp. $\Ker_{\tcut,t}^+ \subseteq \Ima_{\lcut,t}^-$).
\end{itemize}
\end{lem}
\begin{proof}
All four cases are proven by the same argument, which we detail
here in the first case. The Realization Lemma~\ref{lem:realization}
tells us that there exist finite values $x\geq t_x$ and $y\leq t_y$
such that $\Ker_{\rcut, t}^- = \Ker h_t^{(x,t_y)}$ and $\Ima_{\bcut,
  t}^+ = \Ima v_{(t_x,y)}^t$. We then have the following exact
square:
  \[
\xymatrix{
  M_t \ar[r] & M_{(x,t_y)} \\
  M_{(t_x,y)} \ar[r]\ar[u] & M_{(x,y)} \ar[u]
}
\]
The exactness of this square implies that every $\alpha\in \Ker
h_t^{(x,t_y)}$ has a common antecedent 
$\beta\in M_{(t_x,y)}$ with $0\in M_{(x,y)}$, meaning that $\alpha=v_{(t_x,y)}^t(\beta)\in \Ima
v_{(t_x,y)}^t$. Therefore, $\Ker_{\rcut,t}^-\subseteq
\Ima_{\bcut,t}^+$.
The inclusion $\Ker_{\tcut,t}^-\subseteq
\Ima_{\lcut,t}^+$ is obtained symmetrically, with the Realization
 Lemma~\ref{lem:realization} giving some finite $x\leq t_x$ and $y\geq t_y$ such that $\Ker_{\tcut, t}^- = \Ker v_t^{(t_x,y)}$ and $\Ima_{\lcut, t}^+ = \Ima h_{(x,t_y)}^t$. 
\end{proof}
\section{The counting functor}
\label{sec:counting functor}


Our exposition in this section follows~\cite{cb} closely.
To define our functor we consider certain combinations of images and
kernels that, intuitively, capture the elements appearing at the
bottom and left boundaries of a rectangle and that die at its top and
right boundaries.
%
Specifically, given a rectangle $\Rs$ and a point $t\in\Rs$, we define as in~\cite{cb}:
\begin{equation}\label{eq:V+-}
\begin{array}{rl}
V_{\Rs, t}^+ & = \Ima_{\Rs, t}^+ \cap \Ker_{\Rs, t}^+,\\[0.5em]
V_{\Rs, t}^- &= \Ima_{\Rs, t}^+ \cap \Ker_{\Rs, t}^- + \Ima_{\Rs, t}^- \cap \Ker_{\Rs, t}^+
\end{array}
\end{equation}
%
Since $\Ima_{\Rs, t}^- \subseteq \Ima_{\Rs, t}^+$ and $\Ker_{\Rs, t}^- \subseteq \Ker_{\Rs, t}^+$, we have $V_{\Rs, t}^- \subseteq V_{\Rs, t}^+$.
%
Note that these spaces depend a priori on the location of $t$ in the rectangle. In fact, it turns out not to be the case, as the following result shows:

\begin{lem}\label{mit}
Assume $M$ is pfd and satisfies~(\ref{eq:weak_exactness}). Then, for all
$s\leq t\in \Rs$ we have $\rho_s^t(V^\pm_{\Rs, s}) = V^\pm_{\Rs,
  t}$. Furthermore, the induced map $\overline{\rho_s^t}: V_{\Rs, s}^+
/ V_{\Rs, s}^- \to V_{\Rs, t}^+ / V_{\Rs, t} ^-$ is an isomorphism.
\end{lem}

\begin{proof}
This follows from the Transportation Corollary~\ref{cor:transportation}. First, we have:
\begin{align*}
\rho_s^t(V^+_{\Rs, s}) & = \rho_s^t(\Ima^+_{\Rs, s}\cap \Ker^+_{\Rs,s})\\&\subseteq \rho_s^t(\Ima^+_{\Rs, s})\cap \rho_s^t(\Ker^+_{\Rs,s}) = \Ima^+_{\Rs, t}\cap \Ker^+_{\Rs,t} = V^+_{\Rs, t} \\[0.5em]
\rho_s^t(V^-_{\Rs, s}) & = \rho_s^t(\Ima^+_{\Rs, s}\cap \Ker^-_{\Rs,s} + \Ima^-_{\Rs, s}\cap \Ker^+_{\Rs,s}) \\[0.5em]
& \subseteq \rho_s^t(\Ima^+_{\Rs, s})\cap \rho_s^t(\Ker^-_{\Rs,s}) + \rho_s^t(\Ima^-_{\Rs, s})\cap \rho_s^t(\Ker^+_{\Rs,s}) \\[0.5em]
& = \Ima^+_{\Rs, t} \cap \Ker^-_{\Rs,t} + \Ima^-_{\Rs, t} \cap \Ker^+_{\Rs,t} = V^-_{\Rs, t}
\end{align*}
Thus, $\rho_s^t(V^\pm_{\Rs, s}) \subseteq V^\pm_{\Rs, t}$ and the
induced map $\overline{\rho_s^t}$ is well-defined. We will now show
that $\overline{\rho_s^t}$ is both injective and surjective, proving
that $\rho_s^t(V^\pm_{\Rs,s}) = V^\pm_{\Rs,t}$ along the way.

Surjectivity: Take $\beta \in V^+_{\Rs, t} = \Ima_{\Rs, t}^+ \cap
\Ker_{\Rs, t}^+$. Then, $\beta = \rho_s^t(\alpha)$ for some $\alpha \in
\Ima_{\Rs, s}^+$.  Now, $\alpha \in ({\rho_s^t})^{-1}(\beta)
\subseteq ({\rho_s^t})^{-1}(\Ker_{\Rs, t}^+) = \Ker_{\Rs, s}^+$, so
$\alpha\in V^+_{\Rs, s}$. Thus, $\rho_s^t(V^+_{\Rs,
  s}) = V^+_{\Rs, t}$,
which implies that the induced map
$\overline{\rho_s^t}$ is surjective.

Injectivity: Take $\alpha\in V^+_{\Rs, s}$ such that
$\beta=\rho_s^t(\alpha)\in V^-_{\Rs, t}$. Then,
$\beta=\beta_1+\beta_2$ with $\beta_1 \in \Ima_{\Rs, t}^- \cap
\Ker_{\Rs, t}^+$ and $\beta_2 \in \Ima_{\Rs, t}^+ \cap \Ker_{\Rs,
  t}^-$.  By the same argument as before, $\beta_1 =
\rho_s^t(\alpha_1)$ for some $\alpha_1 \in \Ima_{\Rs, s}^- \cap
\Ker_{\Rs, s}^+$. Now, $\rho_s^t(\alpha - \alpha_1) = \beta_2 \in
\Ker_{\Rs, t}^-$ so $\alpha - \alpha_1 \in \Ker_{\Rs, s}^-$.
Moreover, $\alpha-\alpha_1 \in \Ima_{\Rs, s}^+$, so $\alpha \in
V_{\Rs, s}^-$. Thus, $(\rho_s^t|_{V^+_{\Rs,s}})^{-1}(V^-_{\Rs,t})
\subseteq V^-_{\Rs,s}$, which implies that the induced map
$\overline{\rho_s^t}$ is injective. It also implies that
$\rho_s^t(V^-_{\Rs, s}) = V^-_{\Rs, t}$ since we already know that
$\rho_s^t(V^-_{\Rs,s}) \subseteq V^-_{\Rs, t} \subseteq V^+_{\Rs, t} =
\rho_s^t(V^+_{\Rs, s})$.
%
\end{proof}

We can now define an intrinsic quotient, independent of the location
of $t\in \Rs$, by considering the inverse system\footnote{Strictly
  speaking, we use the opposite ordering on~$\Rs$, to fit with the
  convention of~\cite[Chap.~0,
    \textsection13.1]{Grothendieck1961}. This switch in the ordering
  is implicit in the rest of the paper.} of vector spaces $V_{\Rs,
  t}^+ / V_{\Rs, t}^-$ with transition maps $\overline{\rho_s^t}$, and
by taking its inverse limit:
\begin{equation}\label{eq:CR}
C_\Rs(M) = \varprojlim\limits_{t \in \Rs} V_{\Rs, t}^+ / V_{\Rs, t}^-
\end{equation}
By Lemma~\ref{mit}, this limit is isomorphic to
$V_{\Rs,t}^+/V_{\Rs,t}^-$ for all $t\in\Rs$. Moreover, its
construction is entirely functorial, since for any morphism of modules
$\phi:M\to N$ there are canonically induced maps
$\Ima^\pm_{\Rs,t}(M) \to \Ima^\pm_{\Rs,t}(N)$ and $\Ker^\pm_{\Rs,t}(M)
\to \Ker^\pm_{\Rs,t}(N)$, then $V^\pm_{\Rs,t}(M) \to
V^\pm_{\Rs,t}(N)$, then $V^+_{\Rs,t}/V^-_{\Rs,t}(M) \to
V^+_{\Rs,t}/V^-_{\Rs,t}(N)$, and finally $C_\Rs(\phi):C_{\Rs}(M)\to
C_{\Rs}(N)$ by universality of the limit. Thus, $C_\Rs$ is a
functor from the category of pfd bimodules satisfying the equalities
of~(\ref{eq:weak_exactness}) to the category of finite-dimensional
vector spaces. This functor is additive because the inverse limit
commutes whith direct products, and direct products coincide with
direct sums in the category of pfd bimodules.

We refer to $C_\Rs$ as the {\em counting functor} associated to the
rectangle~$R$ because, as we shall see in the following sections, what
it does is, literally, to count the multiplicity of the
summand~$\field_\Rs$ in the direct-sum decomposition of~$M$. In
particular, we can already prove the following fact:
\begin{lem}\label{lem:C_f_mult}
Assume $M$ is pfd and decomposes as a direct sum of rectangle
modules.  Then, for any rectangle~$\Rs$, the multiplicity of the
summand $\field_{\Rs}$ in the direct-sum decomposition of~$M$ is given by~$\dim
C_{\Rs}(M)$.
\end{lem}
\begin{proof}
Since $C_{\Rs}$ is an additive functor, it is enough to prove the result on a single summand $\field_{\Rs'}$. Let us write $\Rs = (\lcut^+ \cap \rcut^-) \times (\bcut^+ \cap \tcut^-)$ and $\Rs' = (\lcut'^+ \cap \rcut'^-) \times (\bcut'^+ \cap \tcut'^-)$.

Suppose first that $\Rs' \neq \Rs$. Then, there is a cut that differs
between $\Rs$ and $\Rs'$, i.e. there is some $\cut\in\{\lcut, \bcut,
\rcut, \tcut\}$ such that $\cut\neq \cut'$. For all $t\in\Rs\cap\Rs'$,
we then have $\bullet^+_{\cut,t}(\field_{\Rs'}) =
\bullet^-_{\cut,t}(\field_{\Rs'})$, where $\bullet$ stands for either
$\Ima$ or $\Ker$ depending on whether $c\in\{\lcut, \bcut\}$ or
$c\in\{\rcut, \tcut\}$. Then, by~(\ref{eq:fs_ima_ker}) we have
$\bullet^+_{\Rs,t}(\field_{\Rs'})=\bullet^-_{\Rs,t}(\field_{\Rs'})$,
which by~(\ref{eq:V+-}) implies that
$V^+_{\Rs,t}(\field_{\Rs'})=V^-_{\Rs,t}(\field_{\Rs'})$ and so
$V^+_{\Rs,t}(\field_{\Rs'})/V^-_{\Rs,t}(\field_{\Rs'})=0$. Meanwhile,
for all $t\in \Rs\setminus\Rs'$, we have $(\field_{\Rs'})_t=0$ and so
$V^+_{\Rs,t}(\field_{\Rs'})/V^-_{\Rs,t}(\field_{\Rs'})=0$. Taking the inverse limit as in~(\ref{eq:CR}), we obtain that
$C_\Rs(\field_{\Rs'})=0$.

Suppose now that $\Rs'=\Rs$. For any $t \in \Rs$ and any $\cut\in\{\lcut,
\bcut, \rcut, \tcut\}$, we have $\bullet^+_{\cut,t}(\field_{\Rs}) =
(\field_{\Rs})_t \simeq \field$ and $\bullet^-_{\cut,t}(\field_{\Rs}) =
0$, where $\bullet$ stands for either $\Ima$ or $\Ker$ depending on
whether $c\in\{\lcut, \bcut\}$ or $c\in\{\rcut, \tcut\}$. Then,
by~(\ref{eq:fs_ima_ker}) we have $\Ima^+_{\Rs,t}(\field_{\Rs}) =
\Ker^+_{\Rs,t}(\field_{\Rs}) = (\field_{\Rs})_t\simeq\field$ while
$\Ima^-_{\Rs,t}(\field_{\Rs}) = \Ker^-_{\Rs,t}(\field_{\Rs}) = 0$,
which by~(\ref{eq:V+-}) implies that
$V^+_{\Rs,t}(\field_{\Rs})=(\field_{\Rs})_t\simeq\field$ while
$V^-_{\Rs,t}(\field_{\Rs})=0$, and so
$V^+_{\Rs,t}(\field_{\Rs})/V^{-}_{\Rs,t}(\field_{\Rs})\simeq\field$,
which, taking the inverse limit as in~(\ref{eq:CR}), gives
$C_\Rs(\field_{\Rs})\simeq\field$ and so $\dim C_\Rs(\field_{\Rs}) =
1$ as desired.
\end{proof}

\section{Submodules}
\label{sec:submodules}

While the counting functor allows us to retrieve the
multiplicity of a summand given a decomposition of~$M$, it is not
enough to prove the existence of such a decomposition.  For this we
need to further assume that $M$ is exact, and to specify a
submodule~$\MRs_\Fs$ of~$M$ for each block $\Fs$, so that in the
following sections we can exhibit an internal direct-sum
decomposition.  To define $\MRs_\Fs$ we use the spaces $V^\pm_{\Fs,t}$
from~(\ref{eq:V+-}) and consider their inverse limits:
\[
V_{\Fs}^{\pm}(M) = \varprojlim\limits_{t \in \Fs} V_{\Fs, t}^{\pm}
\]
%
%
%
\begin{lem}\label{lem:isom_quotients}
  Assume $M$ is pfd and satisfies~(\ref{eq:weak_exactness}). Then, the
  quotient of limits $V^+_\Fs(M)/V^-_\Fs(M)$ is isomorphic to the
  limit of quotients $C_\Fs(M)$.
\end{lem}
\begin{proof}
Recall from Lemma~\ref{mit} that $\rho_s^t(V^-_{\Fs,s}) \subseteq
V^-_{\Fs, t}$ for any $s\leq t\in\Fs$.  Thus, we have an inverse system of
vector spaces $V_{\Fs, t}^-$ for $t \in \Fs$, with transition maps
$\rho_s^t$ for $s\leq t\in\Fs$. This system satisfies the
Mittag-Leffler condition~\cite[Chap.~0, (13.1.2)]{Grothendieck1961}
because every space $V^-_{\Fs,t}$
is finite-dimensional.  Now, since $\Fs$ contains a
countable subset\footnote{Take either
  $\{\min \Fs\}$, or $(\{\min\{t_x \mid t \in\Fs\}\}\times\sQ)\cap \Fs$, or
  $(\sQ\times \{\min\{t_y \mid t \in\Fs\}\})\cap\Fs$, or $\sQ^2 \cap \Fs$,
  depending on whether (respectively) $\Fs$ contains both its left and bottom
  boundaries, or only the left one, or only the bottom one, or none.} that is coinitial for the product order~$\leq$, the
hypotheses of Proposition~13.2.2 of~\cite[Chap.~0]{Grothendieck1961}
hold for the system of exact sequences
\[
0 \rightarrow V_{\Fs, t}^- \rightarrow V_{\Fs, t}^+ \rightarrow V_{\Fs, t}^+/ V_{\Fs, t}^- \rightarrow 0
\]
and hence the limit sequence
%
\[
0 \rightarrow V_{\Fs}^-(M) \rightarrow V_{\Fs}^+(M) \rightarrow
C_\Fs(M)
\rightarrow 0
\]
%
 is also exact, giving the result.
\end{proof}
Letting $\pi_t : V_{\Fs}^+(M) \rightarrow V_{\Fs, t}^+$ denote the
natural map given by the limit, we can make the following
identification:
\[
V_{\Fs}^-(M) = \bigcap\limits_{t \in \Fs} \pi_t^{-1}(V_{\Fs, t}^-) \subseteq V_{\Fs}^+(M)
\]
\begin{lem}\label{rpi}
  Assume $M$ is pfd and satisfies~(\ref{eq:weak_exactness}). Then, for
  all $t \in \Fs$, the induced map $\overline{\pi_t} : V_{\Fs}^+(M) /
  V_{\Fs}^-(M) \rightarrow V_{\Fs,t}^+ / V_{\Fs,t}^-$ is an
  isomorphism.
\end{lem}
\begin{proof}
Following-up the proof of Lemma~\ref{lem:isom_quotients}, for every
$t\in \Fs$ we have the commutative diagram below, where the
vertical arrows are the natural maps given by the limit, and where each
row is known to be exact:
%
\[
\xymatrix{ 0 \ar[r]\ar[d] & V_{\Fs}^-(M)
  \ar[r]\ar^-{\pi_t|_{V_\Fs^-(M)}}[d] & V_\Fs^+(M)
  \ar[r]\ar^-{\pi_t}[d] & C_\Fs(M) \ar[r]\ar^-{\nu_t}[d] & 0
  \ar[d] \\ 0 \ar[r] & V_{\Fs,t}^- \ar[r] & V_{\Fs,t}^+ \ar[r] &
  V_{\Fs,t}^+/V_{\Fs,t}^- \ar[r] & 0 }
\]
%
Hence the (commutative) diagram:
  \[\xymatrix{
    V^+_\Fs(M)/V^-_\Fs(M) \ar^-{\simeq}[r]\ar_-{\overline{\pi_t}}[d] & C_\Fs(M)\ar^-{\nu_t}[d] \\
    V^+_{\Fs,t}/V^-_{\Fs,t} \ar@{=}[r] & V^+_{\Fs,t}/V^-_{\Fs,t}
  }\]
  Now, by Lemma~\ref{mit} the maps $\overline{\rho_s^u}$ are
  isomorphisms for all $s\leq u\in\Fs$, therefore $\nu_t$ itself is an
  isomorphism. The conclusion follows.
\end{proof}
%
%
Now we can identify, for each block~$\Fs$, a submodule $\MRs_{\Fs}$ of
$M$ that is isomorphic to a direct sum of $\dim C_\Fs(M)$ copies of
the interval module~$\field_\Fs$ (see Lemma~\ref{lem:block-copies}
below). Note that, unlike in the 1-d setting~\cite{cb}, not every
vector space complement of~$V_{\Fs}^-(M)$ in $V_{\Fs}^+(M)$ will work
to get a submodule whose support is precisely~$\Fs$. We need to choose
that complement with care so that the submodule does vanish
outside~$\Fs$, and for this we use exactness. Here are the details:
\begin{prop}\label{prop:MRs}
  Assume $M$ is pfd and exact. Then, for each block~$\Fs$ there is a vector
  space complement~$\MRs_{\Fs}^0$ of $V_{\Fs}^-(M)$ in $V_{\Fs}^+(M)$
  such that the family of subspaces defined by
\[
(\MRs_{\Fs})_t = 
\begin{cases}
\pi_t(\MRs_{\Fs}^0)& (t \in \Fs)\\
0& (t \notin \Fs)
\end{cases}
\]
defines a submodule $\MRs_{\Fs}$ of $M$.
\end{prop}
\begin{proof}
Given a fixed  block~$\Fs$, let us first observe that, whatever choice of subspace $\MRs_\Fs^0$ we make such that $V_{\Fs}^+(M) = \MRs_{\Fs}^0 \oplus
V_{\Fs}^-(M)$, the following properties will be satisfied:
\begin{itemize}
\item For any $s \leq t$ sitting both in~$\Fs$, $\rho_s^t((\MRs_{\Fs})_s)\subseteq (\MRs_\Fs)_t$. This is because  $\rho_s^t \circ \pi_s =
  \pi_t$ by definition of $\pi$.
  \item For any $s\leq t$ such that $s\notin\Fs$ and $t\in\Fs$,  $\rho_s^t((\MRs_\Fs)_s) = \rho_s^t(0) = 0 \subseteq
    (\MRs_\Fs)_t$.
\end{itemize}
There only remains to show that, for a suitable choice of subspace~$\MRs_\Fs^0$, we also have $\rho_s^t(\pi_s(\MRs_\Fs^0)) = 0$ for all $s\leq t$ with $s\in\Fs$ and $t\notin\Fs$.
For this we let $\Fs=(\lcut^+\cap \rcut^-)\times(\bcut^+\cap\tcut^-)$ and we distinguish between the various block types:

\medskip

\paragraph*{\bf Case $\Fs$ is a birth quadrant (possibly with $\lcut^-=\emptyset$ or $\bcut^-=\emptyset$).}

Then any choice of vector space complement~$\MRs_\Fs^0$ works
trivially, because there are no indices $s\leq t$ with $s\in\Fs$ and
$t\notin\Fs$.

\medskip

\paragraph*{\bf Case $\Fs$ is a death quadrant and not a band ($\rcut^+\neq\emptyset\neq\tcut^+$).}

Then we will enforce $\pi_s(\MRs_\Fs^0)\subseteq
\Ker^+_{\rcut,s}\cap\Ker^+_{\tcut,s}$ for every $s\in\Fs$, which will
imply that $\rho_s^t(\pi_s(\MRs_\Fs^0)) \subseteq
\rho_s^t(\Ker^+_{\rcut,s}\cap \Ker^+_{\tcut,s}) = 0$ for every $t\geq
s$ with $t\notin\Fs$.  Let then $K^+_{\Fs,s} = \Ker^+_{\rcut,s}\cap
\Ker^+_{\tcut,s}$ for each $s\in\Fs$, and consider the inverse system
formed by these vector spaces with the transition maps $\rho_s^u$ for
$s\leq u\in\Fs$. Since $K^+_{\Fs,s}\subseteq \Ima^+_{\Fs,s}$ by
Lemma~\ref{b-ki}, we have $K^+_{\Fs,s}\subseteq V^+_{\Fs,s}$ and so
the inverse limit $K^+_\Fs(M)$ of the system can be identified as
follows:
\[
K^+_\Fs(M) = \varprojlim\limits_{s \in \Fs} K^+_{\Fs, s}
= \bigcap\limits_{s \in \Fs} \pi_s^{-1}(K^+_{\Fs, s}) \subseteq V_{\Fs}^+(M)
\]
We claim that $V^-_\Fs(M) + K^+_\Fs(M)  = V^+_\Fs(M)$. Indeed, this equality holds at every index $s\in\Fs$ because $K^+_{\Fs,s}\subseteq \Ima^+_{\Fs,s}$:
\[
V^+_{\Fs,s} = \Ima^+_{\Fs,s} \cap (\Ker^-_{\Fs,s} + K^+_{\Fs,s}) = \Ima^+_{\Fs,s}\cap \Ker^-_{\Fs,s} + K^+_{\Fs,s} = V^-_{\Fs,s} + K^+_{\Fs,s}.
\]
In other words, at every index $s\in\Fs$ we have the following exact sequence:
\[
\xymatrix@C=25pt{
  0 \ar[r] &  V^-_{\Fs,s} \cap K^+_{\Fs,s} \ar^-{\alpha\mapsto(\alpha,-\alpha)}[rr] &&
  V^-_{\Fs,s} \oplus K^+_{\Fs,s} \ar^-{(\alpha,\beta)\mapsto\alpha+\beta}[rr] &&
  V^+_{\Fs,s} \ar[r] & 0
}
\]
Since every space $V^-_{\Fs,s} \cap K^+_{\Fs,s}$
is finite-dimensional, the Mittag-Leffler condition is satisfied by
this system of exact sequences, and so, by Proposition~13.2.2
of~\cite[Chap.~0]{Grothendieck1961}, the limit sequence is
exact. After noticing that $\varprojlim\limits V^-_{\Fs,s}\cap K^+_{\Fs,
  s} = V^-_{\Fs}(M)\cap K^+_{\Fs}(M)$ inside~$V^+_{\Fs}(M)$, and that
the canonical morphism  $V^-_{\Fs}(M)\oplus K^+_{\Fs}(M)\to \varprojlim\limits V^-_{\Fs,s}\oplus K^+_{\Fs, s}$ is an isomorphism, we obtain the following exact
sequence:
\[
\xymatrix@C=20pt{
  0 \ar[r] & V^-_{\Fs}(M)\cap K^+_{\Fs}(M) \ar^-{\alpha\mapsto(\alpha,-\alpha)}[rr] &&
  V^-_{\Fs}(M) \oplus K^+_{\Fs}(M) \ar^-{(\alpha,\beta)\mapsto\alpha+\beta}[rr] &&
  V^+_{\Fs}(M) \ar[r] & 0
}
\]
which implies that $V^-_\Fs(M) + K^+_\Fs(M) = V^+_\Fs(M)$, as claimed
earlier. We can then choose\footnote{This is done via a finite
  induction since $V^+_\Fs(M)/V^-_\Fs(M)$ is finite-dimensional by
  Lemma~\ref{rpi}.}  our vector space complement $\MRs_\Fs^0(M)$
inside $K^+_\Fs(M)$, which ensures that $\pi_s(\MRs_\Fs^0)\subseteq
K^+_{\Fs,s}$ for every $s\in\Fs$.

\medskip

\paragraph*{\bf Case $\Fs$ is a horizontal band and not a birth quadrant ($\tcut^+\neq\emptyset$).}

Then we will enforce $\pi_s(\MRs_\Fs^0)\subseteq \Ima^+_{\bcut,s}\cap\Ker^+_{\tcut,s}$ for every $s\in\Fs$, which will imply that $\rho_s^t(\pi_s(\MRs_\Fs^0)) \subseteq \rho_s^t(\Ker^+_{\tcut,s}) = 0$ for every $t\geq s$ with $t\notin\Fs$.  Let then $K^+_{\Fs,s} = \Ima^+_{\bcut,s}\cap \Ker^+_{\tcut,s}$ for each $s\in\Fs$. We have:
\[
V^+_{\Fs,s} = \Ima^+_{\Fs,s}\cap \Ker^+_{\Fs,s}= \Ima^+_{\lcut,s}\cap\Ima^+_{\bcut,s}\cap (\Ker^-_{\rcut,s} + \Ker^+_{\tcut,s}).
\]
Since $\Ker^-_{\rcut,s}\subseteq \Ima^+_{\bcut,s}$ and $\Ker^-_{\tcut,s}\subseteq \Ker^+_{\tcut,s}\subseteq \Ima^+_{\lcut,s}$ by Lemma~\ref{b-ki}, we get:
\[
V^+_{\Fs,s} = \Ima^+_{\lcut,s}\cap \Ker^-_{\rcut,s} +  \Ima^+_{\bcut,s}\cap \Ker^+_{\tcut,s} = \Ima^+_{\lcut,s}\cap \Ker^-_{\rcut,s} + K^+_{\Fs,s}.
\]
Meanwhile, we have:
\begin{align*}
  V^-_{\Fs,s} &= \Ima^+_{\Fs,s}\cap \Ker^-_{\Fs,s} + \Ima^-_{\Fs,s}\cap \Ker^+_{\Fs,s}\\
 & = \Ima^+_{\lcut,s}\cap\Ima^+_{\bcut,s}\cap (\Ker^-_{\rcut,s} + \Ker^-_{\tcut,s}) + 
\Ima^-_{\Fs,s}\cap \Ker^+_{\Fs,s}\\  
  & = \Ima^+_{\lcut,s}\cap \Ker^-_{\rcut,s} + \Ima^+_{\bcut,s}\cap\Ker^-_{\tcut,s} + 
\Ima^-_{\Fs,s}\cap \Ker^+_{\Fs,s} \supseteq \Ima^+_{\lcut,s}\cap \Ker^-_{\rcut,s}.
\end{align*}
Hence, $V^-_{\Fs,s} + K^+_{\Fs,s} = V^+_{\Fs,s}$. By the same argument as in the previous case, we deduce that the limits satisfy  $V^-_\Fs(M) + K^+_\Fs(M)  = V^+_\Fs(M)$. We can then choose our vector space complement $\MRs_\Fs^0$ inside $K^+_\Fs(M)$, which ensures that $\pi_s(\MRs_\Fs^0)\subseteq K^+_{\Fs,s}$ for every $s\in\Fs$.

\medskip

\paragraph*{\bf Case $\Fs$ is a vertical band and not a birth quadrant ($\rcut^+\neq\emptyset$).}

This case is symmetric to the previous one.
\end{proof}
\begin{rem}\label{rem_rpi}
Since we have chosen $\MRs_{\Fs}^0$ such that  $V_{\Fs}^+(M) = V_{\Fs}^-(M) \oplus \MRs_{\Fs}^0$, for every $t\in\Fs$ we have $\MRs_{\Fs}^0 \stackrel{\pi_t}{\simeq}  (\MRs_{\Fs})_t$ and $V_{\Fs, t}^+ = V_{\Fs,t}^- \oplus (\MRs_{\Fs})_t$ by Lemma~\ref{rpi}.
\end{rem}
Assuming that $M$ is pfd and decomposes as a direct sum of block
modules (hence exact), we have by construction and
Lemma~\ref{lem:C_f_mult} that $\dim \MRs_{\Fs}^0 = \dim C_\Fs(M)$ is
equal to the multiplicity of the summand $\field_\Fs$ in the
decomposition of~$M$. More generally:
\begin{lem}\label{lem:block-copies}
Assume that $M$ is pfd and exact. Then, for every block~$\Fs$,
$\MRs_{\Fs}$ is isomorphic to the direct sum of $\dim C_\Fs(M)$ copies
of the block module~$\field_{\Fs}$.
\end{lem}
\begin{proof}
For every $t\in\Fs$, the restriction of $\pi_t$ to
$\MRs_{\Fs}^0$ is an isomorphism onto~$(\MRs_\Fs)_t$ by Lemma~\ref{rpi}.
Take a (finite) basis $\Gamma$ of $\MRs_{\Fs}^0$. For each $\gamma\in \Gamma$,
the elements $\pi_t(\gamma)$ for $t\in\Fs$ are non-zero and they
satisfy $\rho_s^t(\pi_s(\gamma)) = \pi_t(\gamma)$ for all $s\leq t\in\Fs$,
so they span a submodule $N(\gamma)$ of $\MRs_{\Fs}$ that is
isomorphic to $\field_{\Fs}$. Now, for all $t\in\Fs$ the family
$\{\pi_t(\gamma)\}_{\gamma\in \Gamma}$ is a basis of $(\MRs_{\Fs})_t$,
so $\MRs_{\Fs} = \bigoplus_{\gamma\in \Gamma} N(\gamma)$. Finally, the
size of the basis $\Gamma$ is $\dim \MRs_\Fs^0 = \dim C_\Fs(M)$.
\end{proof}

%

\section{Sections and direct sum}
\label{sec:direct_sum}

In this section we assume that $M$ is pfd and exact, and we show that
the submodules $\MRs_{\Fs}$ introduced in Proposition~\ref{prop:MRs}
are in direct sum.
For this we introduce the so-called
{\em sections} associated to these submodules.  A {\em section}
in a vector space~$U$ is a pair $(F^-, F^+)$ of subspaces such that $F^-
\subseteq F^+ \subseteq U$.  We say that a family of sections
$\{(F_\lambda^-, F_\lambda^+)\}_{\lambda \in \Lambda}$ in~$U$ is {\em
  disjoint} if for all $\lambda\neq\mu$, either $F^+_\lambda\subseteq
F^-_\mu$ or $F^+_\mu\subseteq F^-_\lambda$.
Disjointness is a useful concept for proving direct sums thanks to the
following result:
\begin{lem}[\cite{cb}]\label{disjoint-decomp}
Suppose that $\{(F_\lambda^-, F_\lambda^+)\}_{\lambda \in \Lambda}$ is
a disjoint family of sections in~$U$. For each $\lambda\in\Lambda$,
let $\MRs_\lambda$ be a subspace with $F^+_\lambda = \MRs_\lambda \oplus
F^-_\lambda$. Then, the family of spaces
$\{\MRs_\lambda\}_{\lambda\in\Lambda}$ is in direct sum.
\end{lem}
In our context, it turns out that the individual families of image and
kernel subspaces induced by horizontal or vertical cuts are disjoint:
\begin{lem}[\cite{cb}]\label{lem:disjoint_1d}
Given a fixed $t\in\sR^2$, each of the families $\{(\Ima_{\lcut,t}^-,
\Ima_{\lcut,t}^+)\}_{\lcut^+\ni t_x}$, $\{(\Ker_{\rcut,t}^-,
\Ker_{\rcut,t}^+)\}_{\rcut^-\ni t_x}$, $\{(\Ima_{\bcut,t}^-,
\Ima_{\bcut,t}^+)\}_{\bcut^+\ni t_y}$, and $\{(\Ker_{\tcut,t}^-,
\Ker_{\tcut,t}^+)\}_{\tcut^-\ni t_y}$ is disjoint in~$\M_t$.
\end{lem}
Moreover, disjoint families can be combined in certain ways to get new disjoint families, for instance:
\begin{lem}[\cite{cb}]\label{disjoint-pres}
If $\cF=\{(F_\lambda^-, F_\lambda^+)\}_{\lambda \in \Lambda}$ and $\cG=\{G_\sigma^-, G_\sigma^+)\}_{\sigma \in \Sigma}$ are two families of
sections in~$U$ such that $\cF$ is disjoint, then the family
\[
\left\{(F_\lambda^- + G_\sigma^-\cap F_\lambda^+,\,
F_\lambda^- + G_\sigma^+\cap F_\lambda^+)\right\}_{(\lambda, \sigma) \in \Lambda \times \Sigma}
\]
is disjoint. 
\end{lem}
However, unlike in the 1-d setting~\cite{cb}, these results do not
suffice to conclude on the direct sum of the submodules $\MRs_\Fs$
 in our 2-d setting, because the full family of sections  $\cV_t = \{(V^-_{\Fs,t},\, V^+_{\Fs,t})\}_{\Fs:\mathrm{block}\ni t}$ is not
 disjoint.
%
\begin{exmp}\label{ex:disjoint}
  Take for $M$ the direct sum of the modules associated with two birth
  quadrants $\Fb_1$ and $\Fb_2$ whose lower-left corners are not
  comparable in the product order on $\sR^2$ (see Figure~\ref{fig:CE}
  for an illustration).
    \begin{figure}[htb]
  \centering
  \includegraphics[scale=0.5]{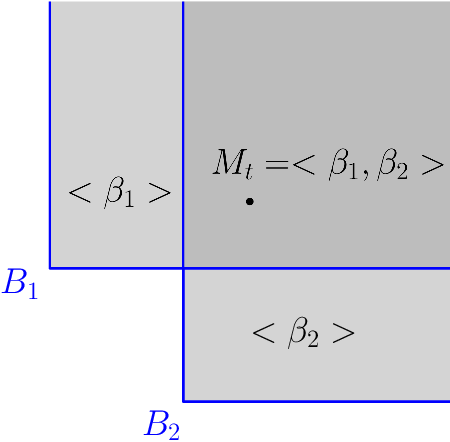}
  \caption{Two incomparable birth quadrants.}
  \label{fig:CE} 
  \end{figure}
    Take $t$ in the intersection of the two
  quadrants. Call $\beta_1$ a generator of the $1$-dimensional
  subspace of $\M_t$ spanned by $\Fb_1$, and call $\beta_2$ a
  counterpart for $\Fb_2$. Then:
  \[
  V^+_{\Fb_1} = \langle \beta_1\rangle \quad\quad 
  V^+_{\Fb_2} = \langle \beta_2\rangle \quad\quad
  V^-_{\Fb_1} = V^-_{\Fb_2} = 0
  \]
  As a result, we have neither $V^+_{\Fb_1}\subseteq V^-_{\Fb_2}$ nor $V^+_{\Fb_2} \subseteq V^-_{\Fb_1}$, which means that the family $\{(V^-_{\Fs,t}, V^+_{\Fs,t})\}_{\Fs:\mathrm{block}\ni t}$ is not disjoint in this example.
\end{exmp}
This lack of global disjointness calls for a special treatment to
establish the direct sum in our 2-d setting. For this we will combine
the above results with some new (more direct) arguments. To start
with, we consider a slightly different family, called $\cF_t=\{
(F_{\Fs,t}^-, F_{\Fs,t}^+)\}_{\Fs:\mathrm{block}\ni t}$, where the spaces
$F^\pm_{\Fs,t}$ are defined as follows:
%
%
\begin{equation}\label{eq:F+-}
\begin{array}{rl}
F_{\Fs,t}^+ &= \Ima_{\Fs,t}^- + V_{\Fs, t}^+ = \Ima_{\Fs,t}^- + \Ker_{\Fs,t}^+ \cap \Ima_{\Fs,t}^+,\\[0.5em]
F_{\Fs,t}^- &= \Ima_{\Fs,t}^- + V_{\Fs, t}^- = \Ima_{\Fs,t}^- + \Ker_{\Fs,t}^- \cap \Ima_{\Fs,t}^+.
\end{array}
\end{equation}
%
As the following lemma shows, we can work indifferently with $\cV_t$ or $\cF_t$ to study the spaces~$(\MRs_{\Fs})_t$:
\begin{lem}\label{lem:FvsV}
$F^+_{\Fs,t} = F^-_{\Fs,t} \oplus (\MRs_{\Fs})_t$ for any block~$\Fs$ and any~$t\in\Fs$.
\end{lem}
\begin{proof}
  From~(\ref{eq:F+-}) and Remark~\ref{rem_rpi} we deduce:
  \[
  F_{\Fs,t}^+ = V_{\Fs,t}^+ + \Ima_{\Fs,t}^- = V_{\Fs,t}^- + (\MRs_{\Fs})_t + \Ima_{\Fs,t}^- = F_{\Fs,t}^- + (\MRs_{\Fs})_t.
  \]
  Meanwhile, we have $(\MRs_{\Fs})_t \subseteq V^+_{\Fs,t}$ therefore:
  \[
  F^-_{\Fs,t} \cap (\MRs_{\Fs})_t = F^-_{\Fs,t} \cap V^+_{\Fs,t} \cap (\MRs_{\Fs})_t
  = V^-_{\Fs,t} \cap (\MRs_{\Fs})_t = 0.
  \]
The result follows.
\end{proof}

The reason for choosing~$\cF_t$ over~$\cV_t$ in our context (as in the
1-d setting) is that it is somewhat easier to work with.
%

The proof that the family of
submodules $\{\MRs_{\Fs}\}_{\Fs:\mathrm{block}}$ is in direct sum is
divided into 2 parts: first, we show that, for each individual block
type, the associated subfamily is in direct sum
(Proposition~\ref{prop:direct_sum_in}); second, we show that the sum
is also direct across block types
(Proposition~\ref{prop:direct_sum_inter}).
\begin{prop}\label{prop:direct_sum_in}
For any fixed block type, the submodules $\MRs_{\Fs}$, where $\Fs$ ranges over the blocks of this type,
are in direct sum.
\end{prop}
\begin{proof}
Submodules are in direct sum if and only if they are such
pointwise. Let then $t\in\sR^2$ be fixed. We focus on each block type
individually:

\medskip

\paragraph*{{\bf Horizontal bands} (including ones that extend to infinity vertically, either upwards or downwards or both).} 

Let $\lcut$ denote the trivial horizontal cut with
$\lcut^-=\emptyset$.  By Lemma~\ref{lem:disjoint_1d},
the family $\{(\Ima_{\tcut,t}^-, \Ima_{\tcut,t}^+)\}_{\tcut^+\ni t_y}$
is disjoint. It follows, by intersecting all the spaces of this family with $\Ima_{\lcut,t}^+$, that $\{(\Ima_{\tcut,t}^-\cap
\Ima_{\lcut,t}^+, \Ima_{\tcut,t}^+\cap\Ima_{\lcut,t}^+)\}_{\tcut^+\ni t_y}$ is also disjoint. By definition this is the same
family as $\{(\Ima_{\Fs,t}^-, \Ima_{\Fs,t}^+)\}_{\Fs:\mathrm{hband}\ni t}$. Then, by Lemma~\ref{disjoint-pres} the family
$\{(F_{\Fs,t}^-, F_{\Fs,t}^+)\}_{\Fs:\mathrm{hband} \ni t}$ itself is
disjoint. Hence, by Lemma~\ref{disjoint-decomp} the family of
subspaces $\{(\MRs_{\Fs})_t\}_{\Fs:\mathrm{hband}\ni t}$ is in direct sum.

\medskip

\paragraph*{{\bf Vertical bands} (including ones that extend to infinity horizontally, either to the left or to the right or both).} 

The treatment is symmetric.

\medskip

\paragraph{{\bf Death quadrants} (including ones that extend to infinity upwards or to the right or both).}

Take any finite family of distinct death quadrants $\Fs_1, \cdots,
\Fs_n$ that contain~$t$. Because they are all distinct,
there must be one of them (say~$\Fs_1$) that is not included
in the union of the others. Hence, there is some $u\geq t$ such
that $u\in\Fs_1\setminus\bigcup_{i>1} \Fs_i$. Now,
suppose there is some relation $\sum_{i=1}^n \alpha_i = 0$ with $\alpha_i\in
(\MRs_{\Fs_i})_t$ nonzero for all $i$. Then, by linearity of~$\rho_t^u$ we
have $\sum_{i=1}^n \rho_t^u(\alpha_i) = 0$. But each $\alpha_i$ with $i>1$ is
sent to zero through $\rho_t^u$ because $u$ lies outside~$\Fs_i$. Hence, $\rho_t^u(\alpha_1)=-\sum_{i=2}^n \rho_t^u(\alpha_i) = 0$. Meanwhile,
we have $\rho_t^u(\alpha_1)\neq 0$ because the restriction of $\rho_t^u$ to
$(\MRs_{\Fs_1})_t$ is injective by Lemma~\ref{lem:block-copies}. This raises
a contradiction.

\medskip

\paragraph{{\bf Birth quadrants} (including ones that extend to infinity downwards or to the left or both).}

All we need to prove is that, for any finite family of distinct birth
quadrants $\Fs_1, \cdots, \Fs_n$, there is at least one of them (say
$\Fs_1$) whose corresponding subspace $(\MRs_{\Fs_1})_t\subseteq \M_t$ is in
direct sum with the ones of the other quadrants in the family. The
result follows then from a simple induction on the size~$n$ of the
family.

Let then $\Fs_1, \cdots, \Fs_n$ be such a family. Each quadrant $\Fs_i$ is bounded to the left by a horinzontal cut $\lcut_i$ and to the bottom by a vertical cut $\bcut_i$. Up to reordering, we can assume that $\Fs_1$ has the rightmost horizontal cut and, in case of ties, it also has the uppermost vertical cut among the quadrants with the same horizontal cut. Formally:
\begin{align*}
\lcut_1^+ & \subseteq \bigcap_{i=2}^n \lcut_i^+\\
\bcut_1^+ & \subseteq \bigcap_{\begin{smallmatrix}i>1\\\lcut_i=\lcut_1\end{smallmatrix}} \bcut_i^+
\end{align*}
It follows that $\Fs_1$ contains none of the other
quadrants. Those can be partitioned into two subfamilies: the ones
(say $\Fs_2, \cdots, \Fs_k$) contain $\Fs_1$ strictly, while the
others ($\Fs_{k+1}, \cdots, \Fs_n$) neither contain~$\Fs_1$ nor are
contained in $\Fs_1$. See Figure~\ref{fig:birth_quadrants} for an
illustration. We analyze the two subfamilies separately.

\begin{figure}[htb]
\centering
\includegraphics[scale=0.5]{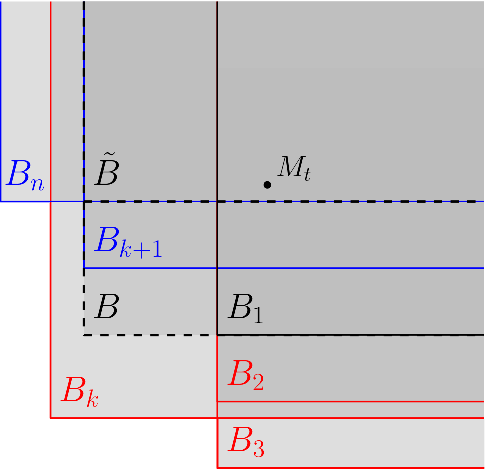}
\caption{Birth quadrants partitioned into two subfamilies}
\label{fig:birth_quadrants} 
\end{figure}

For every $i\in (1, k]$, we have both $\lcut_i^+\supseteq \lcut_1^+$
  and $\bcut_i^+\supseteq \bcut_1^+$, moreover we have either
  $\lcut_i^+\supsetneq \lcut_1^+$ or $\bcut_i^+\supsetneq \bcut_1^+$
  or both. It follows that
  $\Ima_{\lcut_i,t}^+ \subseteq \Ima_{\lcut_1,t}^+$ and
  $\Ima_{\bcut_i,t}^+ \subseteq \Ima_{\bcut_1,t}^+$, moreover either
  $\Ima_{\lcut_i,t}^+ \subseteq \Ima_{\lcut_1,t}^-$ or
  $\Ima_{\bcut_i,t}^+ \subseteq \Ima_{\bcut_1,t}^-$ or both. Hence,
\[
\Ima_{\Fs_i,t}^+ = \Ima_{\lcut_i,t}^+ \cap \Ima_{\bcut_i,t}^+ \subseteq \Ima_{\lcut_1,t}^+ \cap \Ima_{\bcut_1,t}^- + \Ima_{\lcut_1,t}^- \cap \Ima_{\bcut_1,t}^+ = \Ima_{\Fs_1,t}^-.
\]
Summing over $i=2,\cdots, k$ we obtain:
\begin{equation}\label{eq:epsilon}
\sum_{i=2}^k \Ima_{\Fs_i,t}^+ \subseteq \Ima_{\Fs_1,t}^-.
\end{equation}
For every $i\in (k, n]$, we have $\lcut_i^+\supsetneq\lcut_1^+$ and
  $\bcut_i^+\subsetneq\bcut_1^+$.  Let $\tilde\Fs = \bigcap_{i=k+1}^n \Fs_{i}$~---~this birth quadrant neither contains $\Fs_1$ nor is contained
  in~it. Let now $\Fs$ be the smallest quadrant containing both
  $\Fs_1$ and $\tilde\Fs$~---~this quadrant strictly contains them both.
It follows, using the same argument as in the case $i\in (1,k]$:
\begin{equation}\label{eq:zeta}
\Ima_{\Fs_1,t}^+ \cap \left(\sum_{i=k+1}^n \Ima_{\Fs_i,t}^+\right)
\subseteq \Ima_{\Fs_1,t}^+\cap \Ima_{\tilde\Fs,t}^+ =
\Ima_{\Fs,t}^+ \subseteq \Ima_{\Fs_1,t}^-.
\end{equation}
Combining~(\ref{eq:epsilon}) and~(\ref{eq:zeta}), we obtain:
\begin{align*}
(\MRs_{\Fs_1})_t\cap \left(\sum_{i=2}^k (\MRs_{\Fs_i})_t + \sum_{i=k+1}^n (\MRs_{\Fs_i})_t\right)
& \subseteq 
\Ima_{\Fs_1,t}^+ \cap \left(\sum_{i=2}^k \Ima_{\Fs_i,t}^+ + \sum_{i=k+1}^n \Ima_{\Fs_i,t}^+\right) \\
& = \sum_{i=2}^k \Ima_{\Fs_i,t}^+ + \Ima_{\Fs_1,t}^+ \cap \left(\sum_{i=k+1}^n \Ima_{\Fs_i,t}^+\right) \\
& \subseteq \Ima_{\Fs_1,t}^- \subseteq F_{\Fs_1,t}^-,
\end{align*}
which by Lemma~\ref{lem:FvsV} is itself in direct sum with
$(\MRs_{\Fs_1})_t$. Hence the result.
\end{proof}

To establish the direct sum across block types, we adopt the following
convention regarding blocks that belong to more than one type:
\begin{itemize}
\item All the blocks whose support extends to infinity both upwards and to the right are assigned to the birth quadrants.
\item Among the remaining blocks, the ones whose support extends to infinity upwards are assigned to the vertical bands, while the ones whose support extends to infinity to the right are assigned to the horizontal bands.
\end{itemize}

\begin{prop}\label{prop:direct_sum_inter}
Under the previous convention, the submodules $\displaystyle\bigoplus_{\Fs:\mathrm{bquad}} \MRs_{\Fs}$,
$\displaystyle\bigoplus_{\Fs:\mathrm{vband}} \MRs_{\Fs}$, $\displaystyle\bigoplus_{\Fs:\mathrm{hband}} \MRs_{\Fs}$ and $\displaystyle\bigoplus_{\Fs:\mathrm{dquad}}
\MRs_{\Fs}$ are in direct sum.
\end{prop}
\begin{proof}
Again, we only need to prove the direct sum pointwise. Let then
$t\in\sR^2$ be fixed. We order the block types as follows: birth
quadrants, vertical bands, horizontal bands, death quadrants. We will
prove that the summands of each block type are in direct sum with the
summands of the following block types in the sequence.

\medskip

\paragraph*{\bf Birth quadrants.}

Suppose that
\[
\left(\bigoplus\limits_{\Fs:\mathrm{bquad}\ni t} (\MRs_{\Fs})_t\right) \cap
 \left(
\bigoplus\limits_{\Fs:\mathrm{vband}\ni t} (\MRs_{\Fs})_t
+
\bigoplus\limits_{\Fs:\mathrm{hband}\ni t} (\MRs_{\Fs})_t
+
\bigoplus\limits_{\Fs:\mathrm{dquad}\ni t} (\MRs_{\Fs})_t
\right) \neq 0,
\]
where by our convention we treat all the blocks extending to infinity
both upwards and to the right as birth quadrants.  Take then a nonzero
vector~$\alpha$ in the intersection. It can be written as a linear
combination of nonzero vectors $\alpha_1, \cdots, \alpha_n$ taken from the
summands of finitely many birth quadrants $\Fs_1, \cdots, \Fs_n$,
but also as a linear combination of nonzero vectors $\beta_1, \cdots, \beta_m$
taken from the summands of finitely many blocks $\Fs'_1, \cdots,
\Fs'_m$ of other types:
$
\sum_{i=1}^{n} \alpha_i = \alpha = \sum_{j=1}^{m} \beta_j.
$

Pick a point $u\geq t\in\sR^2$ that is large enough so that it lies
outside the  blocks $\Fs'_1, \cdots, \Fs'_m$. Such a
point~$u$ exists because, by our convention, none of the blocks
$\Fs'_1, \cdots, \Fs'_m$ extends to infinity both upwards and to the
right. Meanwhile, $u$ still lies in the  birth
quadrants $\Fs_1, \cdots, \Fs_n$. Let us then consider the image of
$\alpha$ in $\M_{u}$ through the map $\rho_t^u$. On the one hand it is zero
since $\rho_t^u(\beta_j)=0$ for all $j$. On the other hand it is nonzero
since the restriction of $\rho_{t}^u$ to $\bigoplus_{i=1}^n (\MRs_{\Fs_i})_t$
is injective by Lemma~\ref{lem:block-copies} and Proposition~\ref{prop:direct_sum_in}. This raises a contradiction.

\medskip

\paragraph*{\bf Vertical bands.}

Suppose that
\[
\left(\bigoplus\limits_{\Fs:\mathrm{vband}\ni t} (\MRs_{\Fs})_t\right)
\cap
\left(
\bigoplus\limits_{\Fs:\mathrm{hband}\ni t} (\MRs_{\Fs})_t
+
\bigoplus\limits_{\Fs:\mathrm{dquad}\ni t} (\MRs_{\Fs})_t
\right) \neq 0,
\]
where by our convention we treat the blocks extending to infinity
upwards as vertical bands\footnote{The horizontal bands extending to
  infinity upwards have already been treated as birth quadrants.}.
Then we can reproduce the previous reasoning: take a nonzero
vector~$\alpha$ in the intersection, and decompose it as a sum of finitely
many nonzero vectors taken from the summands of vertical bands on the
one hand, as a sum of finitely many nonzero vectors taken from the
summands of horizontal bands or death quadrants on the other
hand. Pick then a point $u\geq t$ with $u_x=t_x$ and with $u_y$ large
enough so that $u$ lies outside all the horizontal bands and death
quadrants involved in the decomposition of~$\alpha$. By looking at the
image $\rho_t^u(\alpha)\in \M_u$ we can raise the same contradiction as
before.

\medskip

\paragraph*{\bf Horizontal bands.}

They are treated symmetrically to the vertical bands.
\end{proof}

\section{Sections and covering}
\label{sec:sections}

In this section and the next we still assume that $M$ is pfd and exact, and we want to show that the
submodules $\MRs_{\Fs}$ introduced in Proposition~\ref{prop:MRs}, for $\Fs$
ranging over all blocks, cover the whole module~$M$, which will conclude
the proof of Theorem~\ref{thm_key}. For this we work again with the
family of sections $\cF_t=\{(F^-_{\Fs, t}, F^+_{\Fs, t})\}_{\Fs:\mathrm{block}\ni t}$
defined in~(\ref{eq:F+-}), and we use the concept of a {\em covering} family,
borrowed from~\cite{cb}.
Given a vector space~$U$, $\{(F_\lambda^-, F_\lambda^+)\}_{\lambda\in\Lambda}$ {\em covers} $U$ if for every proper subspace
$X\subsetneq U$ there is a $\lambda\in\Lambda$ such that
\[
X + F_\lambda^- \neq X + F_\lambda^+.
\]
We say this family {\em strongly covers} $U$ if for all subspaces $X
 \subsetneq U$ and $Z \not\subseteq X$ there is a $\lambda\in\Lambda$ such
that
\[
X + (F_\lambda^- \cap Z) \neq X + (F_\lambda^+ \cap Z).
\]
The use of covering sections is justified by the following result
from~\cite{cb}:
\begin{lem}\label{covering-decomp}
Suppose that $\{(F_\lambda^-, F_\lambda^+)\}_{\lambda \in \Lambda}$ is
a family of sections that covers~$U$. For each $\lambda\in\Lambda$,
let $\MRs_\lambda$ be a subspace with $F^+_\lambda = \MRs_\lambda \oplus
F^-_\lambda$. Then, $U = \sum_{\lambda\in\Lambda}
\MRs_\lambda$. 
\end{lem}
Given a fixed index $t$, the proof in the 1-d
case~\cite{cb} proceeds by showing that the families of kernels and of
images, namely $\{(\Ker^-_{I,t},
\Ker^+_{I,t})\}_{I:\mathrm{interval}\ni t}$ and $\{(\Ima^-_{I,t},
\Ima^+_{I,t})\}_{I:\mathrm{interval}\ni t}$, both strongly
cover~$M_t$. It then applies the following result to deduce that the
family~$\{(F^-_{I,t}, F^+_{I,t})\}_{I:\mathrm{interval}\ni t}$
covers~$M_t$:
\begin{lem}\label{covering-pres}
If $\{(F_\lambda^-, F_\lambda^+)\}_{\lambda \in \Lambda}$ is a family
of sections that covers~$U$, and $\{G_\sigma^-, G_\sigma^+)\}_{\sigma
  \in \Sigma}$ is a family of sections that strongly covers $U$, then
the following family covers~$U$:
\[
\left\{(F_\lambda^- + G_\sigma^-\cap F_\lambda^+,\,
F_\lambda^- + G_\sigma^+\cap F_\lambda^+)\right\}_{(\lambda, \sigma) \in \Lambda \times \Sigma}.
\]
\end{lem}
From there the proof  concludes using Lemma~\ref{covering-decomp}.
Unfortunately, this strategy does not work in our 2-d setting, where
the family of images $\{(\Ima^-_{\Fs,t},
\Ima^+_{\Fs,t})\}_{\Fs:\mathrm{block}\ni t}$ may not strongly
cover~$M_t$, while the family of kernels $\{(\Ker^-_{\Fs,t},
\Ker^+_{\Fs,t})\}_{\Fs:\mathrm{block}\ni t}$ may not even cover it, as shown in the
following two examples:
\begin{exmp}\label{ex:strongcover}
  Consider the module~$M$ from Example~\ref{ex:disjoint}. Take $X = 0$,
  and for $Z$ take the linear span of $\beta_1+\beta_2$ in~$M_t$.
  Since $Z$ is $1$-dimensional, for any block~$\Fs$ such that
  $\Ima_{\Fs,t}^+\cap Z\neq 0$, $\Ima_{\Fs,t}^+$ must contain at least
  $Z$, therefore $\Fs$ must be included in $\Fb_1\cap \Fb_2$. But then
  $\Ima_{\Fs,t}^-=\Ima_{\Fs,t}^+=\M_t$, which means that $X+
  \Ima_{\Fs,t}^-\cap Z = Z = X+\Ima_{\Fs,t}^+\cap Z$. Hence, the
  family~$\{(\Ima^-_{\Fs,t}, \Ima^+_{\Fs,t})\}_{\Fs:\mathrm{block}\ni
    t}$ does not strongly cover~$\M_t$ in this example.
\end{exmp}
\begin{exmp}\label{ex:cover}
  Take for $M$ the direct sum of the modules associated with two death
  quadrants $\Fs_1$ and $\Fs_2$ whose upper-right corners are not
  comparable in the product order on $\sR^2$ (see Figure~\ref{fig:CEC}
  for an illustration).
    \begin{figure}[htb]
  \centering
  \includegraphics[scale=0.5]{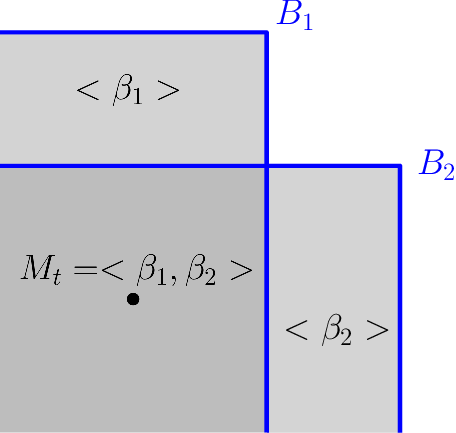}
  \caption{Two incomparable death quadrants.}
  \label{fig:CEC} 
  \end{figure}
    Let $t\in \Fs_1\cap \Fs_2$, and call $\beta_1$ (resp.~$\beta_2$) a
    generator of the 1-dimensional subspace of~$M_t$ spanned by
    $\Fs_1$ (resp. by $\Fs_2$). Take $X$ to be the linear span of
    $\beta_1+\beta_2$. Then, for any choice of block~$\Fs = (\lcut^+
    \cap \rcut^-)\times(\bcut^+\cap \tcut^-)$ containing~$t$, we have
    $\Ker^+_{\rcut, t}\in\{0, \langle \beta_1\rangle, M_t\}$ while
    $\Ker^+_{\tcut, t}\in\{0, \langle \beta_2\rangle, M_t\}$. This
    implies that $\Ker^+_{\rcut, t}\cap \Ker^+_{\tcut, t} \neq 0$ (and
    thus that $\Ker^+_{\Fs, t}$ is potentially different from
    $\Ker^-_{\Fs,t}$) only when $\Ker^+_{\rcut, t} = M_t$ or
    $\Ker^+_{\tcut, t} = M_t$.
    But then we have $\Ker^\pm_{\Fs,t} \in \{\langle
    \beta_1\rangle, \langle \beta_2\rangle, M_t\}$, which implies that
    $X+\Ker^-_{\Fs, t} = M_t = X+\Ker^+_{\Fs, t}$. Hence, the
    family~$\{(\Ker^-_{\Fs,t}, \Ker^+_{\Fs,t})\}_{\Fs:\mathrm{block}\ni t}$ does not
    cover~$\M_t$ in this example.
\end{exmp}
The lack of covering by kernels is particularly
troublesome in our setting. Since it impacts essentially the way we analyze
death quadrants, for now we will leave the contribution of the death
quadrants aside and focus on the rest of the blocks. For this we
introduce the following submodules $L,N$ of~$M$, defined by:
\[
L_t = F^+_{\sR^2,t} = V^+_{\sR^2,t} = \Ima^+_{\sR^2,t}  \quad \quad
N_t = F^-_{\sR^2,t} = V^-_{\sR^2,t} = \Ima^+_{\sR^2,t} \cap \Ker^-_{\sR^2,t} \subseteq L_t
\]
Intuitively, the submodule $N$ should be the one spanned by the block
modules associated to death quadrants that are not $\sR^2$
itself. Here we prove the following result, deferring the analysis
of~$N$ to the next section:
\begin{prop}\label{bigcover}
  For $\Fs$ ranging over the birth quadrants (including~$\sR^2$) and the {\em strict bands} (i.e. bands that are not quadrants, noted {\rm sband} for short), we have:
  \begin{equation}\label{eq:bigcover}
    M = N\ +\ \bigoplus_{\begin{smallmatrix}\Fs:\ \mathrm{bquad}\\\ \mathrm{or\ sband}\end{smallmatrix}} \MRs_{\Fs}
    \end{equation}
\end{prop}
%
%
%
To prove this result we use the strong covering property of
image and kernel sections
in the 1-d setting:
%
\begin{lem}[\cite{cb}]\label{lem:strong-cover_1d}
  Given a fixed $t\in\sR^2$, for any subsets $X\subsetneq \M_t$ and
  $Z\nsubseteq X$, there is a horizontal cut $\lcut$ with $t_x\in
  \lcut^+$ such that $\Ima_{\lcut,t}^-\cap Z \subseteq X \nsupseteq
  \Ima_{\lcut,t}^+\cap Z$. Similarly, there is a vertical cut $\bcut$
  with $t_y\in \bcut^+$ such that $\Ima_{\bcut,t}^-\cap Z \subseteq X
  \nsupseteq \Ima_{\bcut,t}^+\cap Z$. Same for kernels.
\end{lem}
The proof of the 1-d analogue of Proposition~\ref{bigcover} (Lemma~6.1 in~\cite{cb}) proceeds by contradiction: assuming there is an index~$t$ such that $N_t + \sum_{\Fs} (\MRs_{\Fs})_t\subsetneq M_t$, it applies Lemma~\ref{lem:strong-cover_1d} to exhibit some interval $\Fs'$ (the 1-d analogue of a block) such that $F^-_{\Fs',t} \subseteq N_t + \sum_{\Fs} (\MRs_{\Fs})_t \nsupseteq F^+_{\Fs',t}$, then it deduces by Lemma~\ref{lem:FvsV}  that $(\MRs_{\Fs'})_t \nsubseteq N_t + \sum_{\Fs} (\MRs_{\Fs})_t$, a contradiction. Here we follow the same approach, applying Lemma~\ref{lem:strong-cover_1d} in each dimension sequentially, then combining the results using the exactness of~$M$ (via Lemma~\ref{b-ki}) in order to exhibit a block~$\Fs'$ that yields the same contradiction.
\begin{proof}[Proof of Proposition~\ref{bigcover}]
  Given a fixed $t\in\sR^2$, let $X = N_t + \sum_{\Fs} (\MRs_{\Fs})_t$, where $\Fs$ ranges over those birth quadrants (including $\sR^2$) and strict bands that contain~$t$. Suppose for a contradiction that $X \subsetneq \M_t$. Then, apply
Lemma~\ref{lem:strong-cover_1d} with $Z=\M_t$ to get a
horizontal cut $\lcut$ such that $t_x\in \lcut^+$ and
$\Ima_{\lcut,t}^- \subseteq X \nsupseteq \Ima_{\lcut,t}^+$.  Apply then Lemma~\ref{lem:strong-cover_1d} again, this time with $Z = \Ima_{\lcut,t}^+$, to get a
vertical cut $\bcut$ such that $t_y\in \bcut^+$ and $\Ima_{\bcut,t}^-
\cap \Ima_{\lcut,t}^+ \subseteq X \nsupseteq \Ima_{\bcut,t}^+ \cap
\Ima_{\lcut,t}^+$.
Note that the cuts $\lcut, \bcut$ cannot be both trivial, as otherwise (for $\lcut^-=\emptyset=\bcut^-$) we would have
\[
\Ima_{\bcut,t}^+
\cap \Ima_{\lcut,t}^+ = \Ima^+_{\sR^2, t} = F^+_{\sR^2,t} = F^-_{\sR^2,t} + (\MRs_{\sR^2})_t = N_t + (\MRs_{\sR^2})_t \subseteq X,
\]
which contradicts what precedes. Hence, either
$\lcut^-\neq\emptyset$, or $\bcut^-\neq\emptyset$, or both. We
distinguish these three cases below:

\medskip

\paragraph*{\bf Case $\lcut^-\neq\emptyset$ and $\bcut^-\neq\emptyset$.}

Let $\Fs'$ be the birth quadrant~$\lcut^+\cap\bcut^+$. Then,
we have $\Ima^+_{\Fs',t} = \Ima^+_{\lcut,t} \cap \Ima^+_{\bcut,t}$, and 
by Lemma~\ref{b-ki} we also have $\Ker^-_{\Fs',t}\subseteq \Ima^-_{\lcut,t} + \Ima^-_{\bcut,t}$, which gives:
 \begin{align*}
 F_{\Fs',t}^+
& = \Ima_{\Fs',t}^- + \Ima_{\Fs',t}^+\cap \M_t
 = \Ima_{\Fs',t}^+ = \Ima^+_{\lcut,t} \cap \Ima^+_{\bcut,t} \nsubseteq X\\
 F_{\Fs',t}^-
& 
 \subseteq \Ima_{\Fs',t}^- + \Ima_{\Fs',t}^+ \cap (\Ima^-_{\lcut,t} + \Ima^-_{\bcut,t})
 = \Ima_{\Fs',t}^-
 \subseteq \Ima^+_{\lcut,t}\cap \Ima^-_{\bcut,t} + \Ima^-_{\lcut,t}
 \subseteq X.
 \end{align*}
 Hence, by Lemma~\ref{lem:FvsV} we have $(\MRs_{\Fs'})_t \nsubseteq X$, which contradicts the definition of~$X$.
 
\medskip

\paragraph*{\bf Case $\lcut^-=\emptyset$ and $\bcut^-\neq\emptyset$.}

By Lemma~\ref{lem:strong-cover_1d}, applied with $Z=\Ima^+_{\lcut,t}\cap\Ima^+_{\bcut,t}$, there is a vertical cut $\tcut$ such that $t\in\tcut^-$ and
\[
\Ima^+_{\lcut,t}\cap\Ima^+_{\bcut,t}\cap\Ker^-_{\tcut,t} \subseteq X \nsupseteq \Ima^+_{\lcut,t}\cap\Ima^+_{\bcut,t}\cap\Ker^+_{\tcut,t}
\]
Let $\Fs'$ be the horizontal band $\sR\times (\bcut^+\cap\tcut^-)$. Lemma~\ref{b-ki} tells us that $\Ker^-_{\tcut,t}\subseteq \Ima^+_{\lcut,t}$ and $\Ker^-_{\rcut,t}\subseteq \Ima^-_{\bcut,t}$, where $\rcut$ is the trivial cut with $\rcut^+=\emptyset$. Then:
\begin{align*}
  \Ima^-_{\Fs',t} &= \Ima^-_{\bcut,t}\cap \Ima^+_{\lcut,t} \subseteq X\\
  \Ima^+_{\Fs',t}\cap\Ker^-_{\Fs',t} &=  \Ima^+_{\lcut,t}\cap\Ima^+_{\bcut,t}\cap(\Ker^-_{\tcut,t}+\Ker^-_{\rcut,t})\\
  &\subseteq \Ima^+_{\lcut,t}\cap\Ima^+_{\bcut,t}\cap(\Ker^-_{\tcut,t}+\Ima^-_{\bcut,t})\\
  &= \Ima^+_{\lcut,t}\cap\Ima^+_{\bcut,t}\cap\Ker^-_{\tcut,t} + \Ima^+_{\lcut,t}\cap\Ima^-_{\bcut,t} \subseteq X\\
  \Ima^+_{\Fs',t}\cap\Ker^+_{\Fs',t} &\supseteq \Ima^+_{\lcut,t}\cap\Ima^+_{\bcut,t}\cap\Ker^+_{\tcut,t} \nsubseteq X
\end{align*}
Hence, $F^-_{\Fs',t}\subseteq X \nsupseteq F^+_{\Fs',t}$, which by Lemma~\ref{lem:FvsV} implies that $(\MRs_{\Fs'})_t \nsubseteq X$, thus contradicting the definition of~$X$.

\medskip

\paragraph*{\bf Case $\lcut^-\neq\emptyset$ and $\bcut^-=\emptyset$.}

This case is symmetric to the previous one and therefore raises the same
contradiction.
\end{proof}

\section{Completion of the proof}
\label{sec:completion}

We will now show that the sum in~(\ref{eq:bigcover}) is direct
(Corollary~\ref{cor:N_direct-sum}) and that $N$ itself decomposes as a
direct sum of block modules (Corollary~\ref{cor:N_decomp}), which will
conclude the proof of Theorem~\ref{thm_key}. For completeness, we will
also show that the submodules $\MRs_{\Fs}$ provide an internal
direct-sum decomposition of~$M$ (Corollary~\ref{cor:internal-ds}).

\medskip

The proof that the sum in~(\ref{eq:bigcover}) is direct proceeds in two steps, given by the following two Lemmas.

\begin{lem}\label{lem:Nds1}
The submodules $\displaystyle \left(N+\bigoplus_{\Fs:\mathrm{sband}}\MRs_\Fs\right)$ and $\displaystyle \bigoplus_{\Fs:\mathrm{bquad}} \MRs_\Fs$ are in direct sum.
\end{lem}
\begin{proof}
  Suppose the contrary, and let $t\in\sR^2$ such that the intersection
  between the two submodules is non-trivial. Then there are $\alpha\in
  N_t$, $\alpha_1\in(\MRs_{\Fs_1})_t$, $\cdots$,
  $\alpha_r\in(\MRs_{\Fs_r})_t$ and
  $\alpha_{r+1}\in(\MRs_{\Fs_{r+1}})_t$, $\cdots$,
  $\alpha_n\in(\MRs_{\Fs_n})_t$ such that $\Fs_1, \cdots, \Fs_r$ are
  strict bands, $\Fs_{r+1}, \cdots, \Fs_n$ are birth quadrants, and we
  have
  \[
  \alpha + \sum_{i=1}^r \alpha_i = \sum_{i=r+1}^n \alpha_i \neq 0.
  \]
  The strict bands being not birth quadrants, there is some $u\geq t$
  such that $u\notin \bigcup_{i=1}^r \Fs_i$. Moreover, since
  $\alpha\in N_t \subseteq \Ker^-_{\sR^2,t}$, we have
  $\alpha=\alpha_h+\alpha_v$ for some $\alpha_h \in \Ker^-_{\rcut,t}$
  and $\alpha_v\in\Ker^-_{\tcut,t}$, where $\rcut, \tcut$ are the
  trivial cuts with $\rcut^+=\tcut^+=\emptyset$. Then, by the
  Realization Lemma~\ref{lem:realization}, there exist finite
  coordinates $x\geq t_x$ and $y\geq t_y$ such that
  $\alpha_h\in\Ker\rho_t^{(x,t_y)}$ and
  $\alpha_v\in\Ker\rho_t^{(t_x,y)}$. Letting $v$ be the
  point with coordinates $(\max\{u_x, x\},\,\max\{u_y,y\})$, we
  obtain:
    \[
    \rho_t^v\left(\alpha+\sum_{i=1}^r \alpha_i\right) =0.
    \]
    But since $\rho_t^v$ restricted to $\bigoplus_{i={r+1}}^n
    (\MRs_{\Fs_i})_t$ is injective (Lemma~\ref{lem:block-copies} and Proposition~\ref{prop:direct_sum_in}), we
    have $\rho_t^v(\sum_{i=r+1}^n \alpha_i) \neq 0$. This raises a
    contradiction.
\end{proof}

\begin{lem}\label{lem:Nds2}
The submodules $N$ and $\displaystyle \bigoplus_{\Fs:\mathrm{sband}}\MRs_\Fs$ are in direct sum.
\end{lem}
\begin{proof}
Suppose once again the contrary, and let $t\in\sR^2$ be such that the two
submodules have a non-trivial intersection. Then there are $\alpha\in N_t$, $\alpha_1\in(\MRs_{\Fs_1})_t$, $\cdots$, $\alpha_n\in(\MRs_{\Fs_n})_t$ such that $\Fs_1, \cdots, \Fs_n$ are strict bands and
\[
\alpha =\sum_{i=1}^n \alpha_i \neq 0.
\]
Assume without loss of generality that $\Fs_1, \cdots, \Fs_r$ are
horizontal bands while $\Fs_{r+1}, \cdots, \Fs_n$ are vertical
bands. Assume also (still without loss of generality) that none of the
$\alpha_i$'s are zero. Since the bands $\Fs_{r+1}, \cdots, \Fs_n$ are
not birth quadrants, there is some point $u$ with $u_x\geq t_x$ and
$u_y=t_y$ such that $u\notin\bigcup_{i=r+1}^n \Fs_i$. Then, we have
$\rho_t^u(\sum_{i=r+1}^n \alpha_i) = 0$, while
$\rho_t^u(\sum_{i=1}^r \alpha_i) \neq 0$ since the restriction of $\rho_t^u$ to $\bigoplus_{i=1}^r(\MRs_{\Fs_i})_t$ is injective by Lemma~\ref{lem:block-copies} and Proposition~\ref{prop:direct_sum_in}. Thus,
\begin{equation}\label{eq:betas}
\beta = \sum_{i=1}^r \beta_i\neq 0,
\end{equation}
where $\beta=\rho_t^u(\alpha)\in N_u$ and each $\beta_i =
\rho_t^u(\alpha_i) \in (\MRs_{\Fs_i})_u$. Now, we have $N_u\subseteq
\Ima^+_{\sR^2,u} = F^+_{\sR^2,u}$. As we saw in the proof of
Proposition~\ref{prop:direct_sum_in} (case of horizontal bands), the
family of sections $\{(F^-_{\sR^2, u},
F^+_{\sR^2,u}),\,(F^-_{\Fs_1,u}, F^+_{\Fs_1,u}),\,\cdots,\,
(F^-_{\Fs_r,u}, F^+_{\Fs_r,u})\}$ is disjoint. Since $(F^-_{\sR^2, u},
F^+_{\sR^2,u})$ is the minimum element in this family, meaning that
$F^+_{\sR^2, u}\subseteq F^-_{\Fs_i,u}$ for every~$i$, the family
$\{(0, F^+_{\sR^2,u}),\,(F^-_{\Fs_1,u}, F^+_{\Fs_1,u}),\,\cdots,\,
(F^-_{\Fs_r,u}, F^+_{\Fs_r,u})\}$ is also disjoint, which by
Lemma~\ref{disjoint-decomp} implies that $F^+_{\sR^2, u}$ is in direct
sum with $\bigoplus_{i=1}^r(\MRs_{\Fs_i})_u$. As a subspace of
$F^+_{\sR^2,u}$, $N_u$ itself is also in direct sum with
$\bigoplus_{i=1}^r(\MRs_{\Fs_i})_u$, hereby
contradicting~(\ref{eq:betas}).
\end{proof}

It follows from Lemmas~\ref{lem:Nds1} and~\ref{lem:Nds2} that the sum in~(\ref{eq:bigcover}) is direct, that is:
\begin{cor}\label{cor:N_direct-sum}
  $\displaystyle M = N\ \oplus\ \bigoplus_{\begin{smallmatrix}\Fs:\ \mathrm{qbirth}\\\ \mathrm{or\ sband}\end{smallmatrix}} \MRs_{\Fs}$.
\end{cor}

Now we show that $N$ decomposes as a direct sum of block
modules. For this we use vector-space duality. Let~$N^*$ be
the pointwise dual of~$N$. Since duality is a contravariant functor,
$N^*$ is a persistence module indexed over $(\sR^\op)^2$, where
$\sR^\op$ denotes the poset $\sR$ with the opposite order $\geq$. 
\begin{lem}\label{lem:N_exact}
  $N^*$ is pfd and exact.
\end{lem}
\begin{proof}
  We first need to show that $N$ itself is pfd and exact. That it is pfd follows immediately from the fact that it is a submodule of the pfd module~$M$. For exactness, let $s\leq t\in\sR^2$ and take an element $\delta\in N_t$ that has preimages $\beta\in N_{(t_x,s_y)}$ and $\gamma\in N_{(s_x, t_y)}$. Then, by exactness of~$M$, $\beta$ and $\gamma$ have a common preimage $\alpha\in M_s$. Let us show that $\alpha\in N_s$.

  First of all, by the Transportation Corollary~\ref{cor:transportation}, we know that $\alpha\in (\rho_s^t)^{-1}(N_t)\subseteq (\rho_s^t)^{-1}(\Ker^-_{\sR^2,t}) = \Ker^-_{\sR^2,s}$. Meanwhile, since $\beta\in N_{(t_x,s_y)}\subseteq \Ima^+_{\sR^2,(t_x,s_y)}$, for any $u\leq s$ with $u_x=s_x$ there is some preimage $\beta_u$ of $\beta$ in $M_{(t_x, u_y)}$, which, by exactness of~$M$, implies that there is some common preimage $\alpha_u$ of $\alpha$ and $\beta_u$ in $M_u$. Thus, $\alpha\in \Ima^+_{\bcut, s}$, where $\bcut$ is the trivial vertical cut with $\bcut^-=\emptyset$. Symmetrically,  and using $\gamma$, we have $\alpha\in \Ima^+_{\lcut,s}$ where $\lcut$ is the trivial horizontal cut with $\lcut^-=\emptyset$. Hence $\alpha$ belongs also to $\Ima^+_{\sR^2,s}$ and therefore to~$N_s$. This means that $N$ is exact.

 Now, duality sends finite-dimensional vector spaces to finite-dimensional vector spaces, hence $N^*$ is pfd just as~$N$. And since duality is an additive and exact functor, for every $s\leq t$ the exact sequence  
\[\xymatrix@C=110pt{
N_s \ar^-{\scriptstyle \phi\,=\,\left(h_s^{(t_x, s_y)},\; v_s^{(s_x, t_y)}\right)}[r] & N_{(t_x, s_y)} \oplus N_{(s_x, t_y)} \ar^-{\psi\,=\,v_{(t_x, s_y)}^t - h_{(s_x, t_y)}^t}[r] & N_t
}\]
associated with the commutative diagram
\[
\xymatrix{
N_{(s_x, t_y)}\ar^-{h_{(s_x, t_y)}^t}[rr] && N_t\\\\
N_s\ar_-{h_s^{(t_x, s_y)}}[rr]\ar^-{v_s^{(s_x, t_y)}}[uu] && N_{(t_x, s_y)}\ar_-{v_{(t_x, s_y)}^t}[uu]
}
\]
turns into the exact sequence
\[
\xymatrix@C=110pt{
N_s^* \ar@{<-}^-{\scriptstyle \phi^*\,=\,{h_s^{(t_x, s_y)}}^* +  {v_s^{(s_x, t_y)}}^*}[r] & N_{(t_x, s_y)}^* \oplus N_{(s_x, t_y)}^* \ar@{<-}^-{\psi^*\,=\,\left({v_{(t_x, s_y)}^t}^*,\, -{h_{(s_x, t_y)}^t}^*\right)}[r] & N_t^*
}\]
which means that the sequence
\[
\xymatrix@C=110pt{
N_s^* \ar@{<-}^-{\scriptstyle \phi^*\,=\,{h_s^{(t_x, s_y)}}^* -  {v_s^{(s_x, t_y)}}^*}[r] & N_{(t_x, s_y)}^* \oplus N_{(s_x, t_y)}^* \ar@{<-}^-{\psi^*\,=\,\left({v_{(t_x, s_y)}^t}^*,\, {h_{(s_x, t_y)}^t}^*\right)}[r] & N_t^*
}\]
associated with the diagram
\[
\xymatrix{
N_{(s_x, t_y)}^*\ar@{<-}^-{{h_{(s_x, t_y)}^t}^*}[rr] && N_t^*\\\\
N_s^*\ar@{<-}_-{{h_s^{(t_x, s_y)}}^*}[rr]\ar@{<-}^-{{v_s^{(s_x, t_y)}}^*}[uu] && N_{(t_x, s_y)}^*\ar@{<-}_-{{v_{(t_x, s_y)}^t}^*}[uu]
}
\]
is also exact. Hence the result.
\end{proof}

We can now apply Corollary~\ref{cor:N_direct-sum} to~$N^*$, and add the following  observation:

\begin{lem}\label{lem:bot}
For all $t\in(\sR^\op)^2$,  $\Ima^+_{(\sR^\op)^2,t}(N^*) = 0$.
\end{lem}
\begin{proof}
  Denote by $X^\bot$ the annihilator of any subspace $X\subseteq N_t$: 
  \[
  X^\bot = \{\phi\in N_t^* \mid \phi(\alpha)=0\ \forall
  \alpha \in X\}.
  \]
  Since taking the annihilator turns sums into intersections and kernels into images, we have:
\[
\Ima^+_{(\sR^\op)^2,t}(N^*) = \left(\Ker^-_{\sR^2,t}(N)\right)^\bot = N_t^\bot =0.
\]
\end{proof}

It follows that $N^*$ decomposes as a direct sum of block modules
indexed over~$(\sR^\op)^2$, and by duality:
\begin{cor}\label{cor:N_decomp}
  $N$ decomposes as a
direct sum of block modules indexed over~$\sR^2$.
\end{cor}
Corollaries~\ref{cor:N_direct-sum} and~\ref{cor:N_decomp} conclude the
proof of Theorem~\ref{thm_key}. Note that one can further show that
$N$ has surjective internal morphisms, so that all the blocks involved
in its decomposition are death quadrants. This is not mandatory for
the proof of the main theorem, however it justifies our previous
intuition that $N$ is the submodule of $M$ spanned by the block
modules associated to death quadrants (except $\sR^2$ itself since
$N_t\subseteq \Ker^-_{\sR^2,t}$).

\medskip

For completeness, we also show (a posteriori, using the Decomposition
Theorem~\ref{thm_key}) that the submodules~$\MRs_{\Fs}$, for $\Fs$
ranging over all blocks, do give an internal direct-sum decomposition
of~$M$.
\begin{cor}\label{cor:internal-ds}
$\displaystyle M = \bigoplus_{\Fs:\mathrm{block}} \MRs_\Fs$.
\end{cor}
\begin{proof}
  By Theorem~\ref{thm_key} we know that $M$ decomposes as follows:
  \[
  M \simeq \bigoplus_{\Fs:\mathrm{block}} \field_\Fs^{n_\Fs},
  \]
where $n_\Fs\in\sN$ denotes the multiplicity of the
summand~$\field_\Fs$ in the decomposition. Then, for each~$\Fs$,
Lemmas~\ref{lem:C_f_mult} and~\ref{lem:block-copies} ensure that
$\MRs_\Fs \simeq \field_\Fs^{n_\Fs}$, while
Propositions~\ref{prop:direct_sum_in} and~\ref{prop:direct_sum_inter}
ensure that the $\MRs_\Fs$'s are in direct sum in~$M$. Hence, at every index
$t\in\sR^2$ we have $M_t \simeq \bigoplus_{\Fs:\mathrm{block}}
(\MRs_\Fs)_t$, and therefore $M_t = \bigoplus_{\Fs:\mathrm{block}}
(\MRs_\Fs)_t$ since $M_t$ is finite-dimensional.
\end{proof}

\section{Applications}
\label{sec:applications}

To conclude the paper, we briefly discuss some of the implications of our
decomposition theorem.

\subsection{Barcodes and stability for exact pfd bimodules}

By Theorem~\ref{thm_key}, to any exact pfd persistence bimodule
$M$ we can associate the multiset of blocks involved in
its decomposition~(\ref{eq:decomp}). This multiset is called the {\em
  barcode} of~$M$ and denoted by~$\barcode(M)$. The following isometry result follows\footnote{Strictly speaking, the result in~\cite{b-shdidpm-16,bl-aszpm-16} is stated for 
bimodules indexed over the open half-plane above the anti-diagonal $x+y=0$. However, a careful look at the proof reveals that the result extends easily to 
bimodules indexed over~$\sR^2$.} then from~\cite{b-shdidpm-16,bl-aszpm-16}:
\begin{cor}\label{cor:stability_pfd}
For any exact pfd persistence bimodules $M$ and $N$ we have:
\[
\disti(M, N) = \dist_b(\barcode(M),\ \barcode(N)),
\]
where $\disti$ denotes the {\em interleaving
  distance} as defined in~\cite{l-oid-11}, and  where $\distb$ denotes the bottleneck distance as defined in~\cite{b-shdidpm-16,bl-aszpm-16}.
\end{cor}

\subsection{Bimodules indexed over a rectangle or the open half-plane above the anti-diagonal}
\label{sec:rectangle}

We can easily adapt our main theorem to decompose bimodules that are
indexed only over a rectangle in the plane, possibly extending to
infinity. The proof is given in
Appendix~\ref{sec:proof_thm_decomp_rect}.
\begin{thm}\label{thm:decomp_rect}
Let~$M$ be a bimodule indexed over some rectangle~$\Rs\subseteq\sR^2$.  If~$M$ is pfd and exact, then it is block-decomposable, more precisely:
\[
M  \simeq \bigoplus_{\Fs\in \barcode(M)} \field_{\Fs\cap\Rs}.
\]
Moreover, the decomposition is unique up to isomorphism and reordering of the terms (note that two blocks~$\Fs, \Fs'$ having the same intersection with~$\Rs$ yield the same indecomposable $\field_{\Fs\cap\Rs}=\field_{\Fs'\cap\Rs}$). 
\end{thm}
We can also adapt the decomposition theorem to decompose bimodues that
are indexed over the open half-plane~$\sU$ above the anti-diagonal:
$\sU = \{t\in\sR^2 \mid t_x+t_y>0\}$. The proof is given in
Appendix~\ref{sec:proof_thm_decomp_U}; it follows the scheme
of~\cite[Section~5]{botnan2018decomposition} but does not make use of
the general decomposition result in~\cite{botnan2018decomposition}.
\begin{thm}\label{thm:decomp_U}
Let~$M$ be a bimodule indexed over~$\sU$.  If~$M$ is pfd and exact,
then it is block-decomposable, more precisely:
\[
M  \simeq \bigoplus_{\Fs\in \barcode(M)} \field_{\Fs\cap\sU}.
\]
Moreover, the decomposition is unique up to isomorphism and reordering of the terms (again, two blocks~$\Fs, \Fs'$ having the same intersection with~$\sU$ yield the same indecomposable $\field_{\Fs\cap\sU}=\field_{\Fs'\cap\sU}$). 
\end{thm}

\subsection{Interlevel-sets persistence}
\label{sec:interlevel-sets}

We now consider the bimodules that arise in the study of
interlevel-sets persistence, which served as the initial motivation
for this work.  Let $\sU=\{t\in\sR^2 \mid t_x+t_y > 0\}$ as in the
previous section. We observe that $\sU$~is naturally identified with
the set of nonempty bounded open intervals of $\sR$ via the following
bijection:
\[
\sR \supset (x,y) \mapsto (-x,y) \in\sU.
\]
Moreover, if we equip $\sU$ with the product order inherited from
$\sR^2$, and the set of bounded open intervals with the inclusion
order, then the above bijection is an isomorphism of posets.

Given now a topological space~$\sX$ and an $\sR$-valued function
$\func:\sX\to\sR$, let $\cS_\func$ denote the {\em interlevel-sets
  filtration} of~$\func$, which assigns the space $(\cS_\func)_s =
\func^{-1}((-t_x,t_y))$ to any point $t\in\sU$. $\cS_\func$ can be
viewed as a functor from the poset $\sU$ to the category of
topological spaces. The composition $\Hgr\circ\cS_\func$ (where $\Hgr$
is a shorthand for singular homology with coefficients in a fixed
field~$\field$) is then a functor from $\sU$ to the category of
$\field$-vector spaces. The map $\func$ is called {\em pfd}
whenever $\Hgr\circ\cS_\func$ is a pfd
module. Note that this module is always exact because, for any $s<
t\in\sU$, the following diagram
\[
\xymatrix{
\func^{-1}((-s_x, t_y)) \ar^-{\subseteq}[r] & \func^{-1}((-t_x, t_y))\\
\func^{-1}((-s_x, s_y))\ar^-{\subseteq}[u]\ar^-{\subseteq}[r] & \func^{-1}((-t_x, s_y))\ar_-{\subseteq}[u]
}\]
induces an exact diagram in homology, by the Mayer-Vietoris theorem.
We then have the following decomposition result as a byproduct of
Theorem~\ref{thm:decomp_U}:
\begin{cor}\label{cor:int-modules_decomp}
  For any topological space $\sX$ and any pfd function $\func:\sX\to \sR$, the bimodule $\Hgr\circ\cS_\func$ is block decomposable, that is:
  \[
  \Hgr\circ\cS_\func \simeq \bigoplus_{\Fs\in \barcode(\Hgr\circ\cS_\func)} \field_{\Fs\cap\sU}.
  \]
Moreover, the decomposition is unique up to isomorphism and reordering of the terms (again, two blocks~$\Fs, \Fs'$ having the same intersection with~$\sU$ yield the same indecomposable $\field_{\Fs\cap\sU}=\field_{\Fs'\cap\sU}$). 
\end{cor}
This result answers a conjecture of Botnan and
Lesnick~\cite[Conjecture~8.3]{bl-aszpm-16}. Combined with
Corollary~\ref{cor:stability_pfd}, it induces a general stability
result for interlevel-sets persistence, in which the functions
considered do not have to be of {\em Morse
  type}~\cite{bl-aszpm-16,cdsm-phzp-16}:
%
\begin{cor}\label{cor:int-modules_stability}
For any pfd functions $\func, \func':\sX\to \sR$, the barcodes
$\barcode(\Hgr\circ\cS_\func)$ and $\barcode(\Hgr\circ\cS_{\func'})$ are
well-defined and we have:
\[
\dist_b(\barcode(\Hgr\circ\cS_\func),\ \barcode(\Hgr\circ\cS_{\func'})) \leq \|\func-\func'\|_\infty.
\]
\end{cor}
%
%
Alternatively, one may consider interlevel-sets filtrations obtained
by taking preimages of bounded \emph{closed} intervals---excluding
singletons
The corresponding bimodules are also indexed over~$\sU$, moreover they
are pfd and exact by the Mayer-Vietoris theorem, therefore they also
decompose as direct sums of block summands.

\subsection{$\sZ^2$-indexed modules}

Assume now that $M$ is indexed over $\sZ^2$ only, while being still
pfd and exact. Then we can extend~$M$ to $\sR^2$ by taking its left or
right Kan extension (say left)~$\bar M$. Since the inclusion
~$\sZ^2\to\sR^2$ is a fully faithful functor between posets, $\bar M$
restricts to a module isomorphic to~$M$ on~$\sZ^2$---see e.g.~\cite[Corollary
  X.3.3]{mac2013categories}. In fact, since every down-set in~$\sR^2$ has a greatest element in $\sZ^2$, $\bar M$ is piecewise constant:
\[
\forall t=(x,y)\in\sR^2,\quad \bar M_t = \varinjlim M|_{\{s\in\sZ^2 \mid s\leq t\}} \simeq M_{(\lfloor x\rfloor, \lfloor y\rfloor)}.
\]
Therefore, $\bar M$ is exact and pfd like $M$, which implies that it
decomposes into block summands, by Theorem~\ref{thm_key}. Then, by
restriction, $M$ itself also decomposes into block summands, which
yields the following result:
\begin{thm}\label{thm_key_Z2}
  For any exact pfd bimodule $M$ indexed over~$\sZ^2$, we have:
  \[
  M\simeq \bigoplus_{\Fs\in \barcode(M)} \field_{\Fs\cap\sZ^2}.
  \]
\end{thm}
Note that the above construction works as well when~$\sZ^2$ is
replaced by~$\sN^2$ or by any finite grid, leading to a similar conclusion.

\subsection{$\sZ$-indexed zigzag modules}

It is well-known that pfd zigzag modules with a countable index set
decompose as direct sums of interval modules~---~see
e.g.~\cite{botnan2015interval} for a recent treatment. This result can
 be obtained as a byproduct of our main theorem. The construction
is similar to the one in~\cite{bl-aszpm-16}:
given such a zigzag module~$M$, assume without loss of generality that
its index set is~$\sZ$, and furthermore that the arrow orientations in
the zigzag are alternating (which can be ensured by inserting in
isomorphisms at the right places):
\[\xymatrix{
\cdots & \ar[l] M_{2i-2} \ar[r] & M_{2i-1} & \ar[l] M_{2i} \ar[r] & M_{2i+1} & \ar[l] M_{2i+2} \ar[r] & \cdots
}\]
Embed then the zigzag as an infinite staircase~$S$ in the plane:
\[
\xymatrix@C=20pt{\cdots& \ar[l] \stackrel{(i-1,-i+1)}{\bullet} \ar[r] & \stackrel{(i,-i+1)}{\bullet} & \ar[l] \stackrel{(i,-i)}{\bullet} \ar[r] & \stackrel{(i+1,-i)}{\bullet} & \ar[l] \stackrel{(i+1,-i-1)}{\bullet} \ar[r] & \cdots
}\]
and realize~$M$ as a representation~$\bar M$ of that staircase by
letting $\bar M_{(i, -i)}=M_{2i}$ and $\bar M_{(i+1, -i)} = M_{2i+1}$
for all $i\in\sZ$. Then, take left and right Kan extensions of~$\bar
M$ to get a bimodule~$\bar{\bar M}$ indexed over~$\sZ^2$. As in the
previous section, this extension restricts to a zigzag module isomorphic
to~$M$ on the staircase. Moreover, since every up-set or down-set
in~$\sZ^2$ intersects only a finite portion of the staircase,
$\bar{\bar M}$ is pfd. Finally, since left (resp. right) Kan
extensions of diagrams of the form
$\xymatrix{\bullet&\ar[l]\bullet\ar[r]&\bullet}$
(resp. $\xymatrix{\bullet\ar[r]&\bullet&\ar[l]\bullet}$) are pushouts
(resp. pullbacks), and since iterated Kan extensions are also Kan
extensions, $\bar{\bar M}$ is isomorphic to a bimodule obtained by taking 
pushouts and pullbacks iteratively, which implies that~$\bar{\bar M}$ is exact. Then,
by
Theorem~\ref{thm_key_Z2} we have
\[
  \bar{\bar M}\simeq \bigoplus_{\Fs\in \barcode(\bar{\bar M})} \field_{\Fs\cap\sZ^2},
\]
and by restriction we deduce
\[
  M\simeq \bigoplus_{\Fs\in \barcode(\bar{\bar M})} \field_{\Fs\cap S},
\]
where each $\Fs\cap S$ is an interval. Hence the result.

\setcounter{section}{0}
\renewcommand\thesection{\Alph{section}}

\section{Equivalence of definitions}
\label{sec:proof-equiv}

Let $\Rs=(\lcut^+\cap \rcut^-) \times (\bcut^+ \cap \tcut^-)$ be a
rectangle, and let $t\in\Rs$.  To prove that the definitions
in~(\ref{eq:fs_ima_ker_bis}) are equivalent to those
in~(\ref{eq:fs_ima_ker}), it is sufficient to show the following
equalities:
\begin{subequations}
  \begin{equation}\label{eq:equivIma+}
    I^+_{\Rs,t} = \Ima^+_{\lcut,t} \cap \Ima^+_{\bcut,t}
  \end{equation}
  \begin{equation}\label{eq:equivIma-}
    I^-_{\Rs,t} = \Ima^-_{\lcut,t} + \Ima^-_{\bcut,t}
  \end{equation}
  \begin{equation}\label{eq:equivKer+}
    K^+_{\Rs,t} = \Ker^+_{\rcut,t} \cap \Ker^+_{\tcut,t}
  \end{equation}
  \begin{equation}\label{eq:equivKer-}
    K^-_{\Rs,t} = \Ker^-_{\rcut,t} + \Ker^-_{\tcut,t}
  \end{equation}
\end{subequations}

\bigskip
\noindent {\bf Proof of~(\ref{eq:equivIma+}):}

\begin{align*}
  I^+_{\Rs,t}
  &= \bigcap_{\begin{smallmatrix}s\in\Rs\\s\leq t\end{smallmatrix}} \Ima \rho_s^t
  = \bigcap_{\begin{smallmatrix}\lcut^+\ni x\leq t_x\\\bcut^+\ni y\leq t_y\end{smallmatrix}} \Ima \rho_{(x,y)}^t
  \stackrel{\mbox{\scriptsize (Eq. \ref{eq:weak_exactness})}}{=} \bigcap_{\begin{smallmatrix}\lcut^+\ni x\leq t_x\\\bcut^+\ni y\leq t_y\end{smallmatrix}} \Ima \rho_{(x,t_y)}^t \cap \Ima \rho_{(t_x,y)}^t \\
  &= \bigcap_{\lcut^+\ni x\leq t_x}\ \bigcap_{\bcut^+\ni y\leq t_y} \Ima \rho_{(x,t_y)}^t \cap \Ima \rho_{(t_x,y)}^t \\
  &= \bigcap_{\lcut^+\ni x\leq t_x} \left(\Ima \rho_{(x,t_y)}^t \cap \bigcap_{\bcut^+\ni y\leq t_y}  \Ima \rho_{(t_x,y)}^t\right) \\
  &= \left(\bigcap_{\lcut^+\ni x\leq t_x} \Ima \rho_{(x,t_y)}^t\right) \cap \left(\bigcap_{\bcut^+\ni y\leq t_y}  \Ima \rho_{(t_x,y)}^t\right)
  = \Ima^+_{\lcut,t} \cap \Ima^+_{\bcut,t}.
\end{align*}

\bigskip
\noindent {\bf Proof of~(\ref{eq:equivIma-}):}

\begin{align*}
  I^-_{\Rs,t}
  &= \sum_{\begin{smallmatrix}s\notin\Rs\\s\leq t\end{smallmatrix}} \Ima \rho_s^t
  = \sum_{\begin{smallmatrix}x\in\lcut^-\ \mathrm{or}\ y\in\bcut^-\\x\leq t_x\ \mathrm{and}\ y\leq t_y\end{smallmatrix}} \Ima \rho_{(x,y)}^t
  \stackrel{\mbox{\scriptsize (Eq. \ref{eq:weak_exactness})}}{=} \sum_{\begin{smallmatrix}x\in\lcut^-\ \mathrm{or}\ y\in\bcut^-\\x\leq t_x\ \mathrm{and}\ y\leq t_y\end{smallmatrix}} \Ima \rho_{(x,t_y)}^t \cap \Ima \rho_{(t_x,y)}^t \\
  &= \sum_{\begin{smallmatrix}x\in\lcut^-\\y\leq t_y\end{smallmatrix}} \Ima \rho_{(x,t_y)}^t \cap \Ima \rho_{(t_x,y)}^t   \ +\  \sum_{\begin{smallmatrix}y\in\bcut^-\\x\leq t_x\end{smallmatrix}} \Ima \rho_{(x,t_y)}^t \cap \Ima \rho_{(t_x,y)}^t\\
  &= \sum_{\begin{smallmatrix}x\in\lcut^-\\y= t_y\end{smallmatrix}} \Ima \rho_{(x,t_y)}^t \cap \Ima \rho_{(t_x,y)}^t   \ +\  \sum_{\begin{smallmatrix}y\in\bcut^-\\x= t_x\end{smallmatrix}} \Ima \rho_{(x,t_y)}^t \cap \Ima \rho_{(t_x,y)}^t\\
  &= \sum_{x\in\lcut^-} \Ima \rho_{(x,t_y)}^t    \ +\  \sum_{y\in\bcut^-} \Ima \rho_{(t_x,y)}^t
  = \Ima^-_{\lcut,t} + \Ima^-_{\bcut,t}.
\end{align*}

\bigskip
\noindent {\bf Proof of~(\ref{eq:equivKer+}):}
it is symmetric to that of~(\ref{eq:equivIma-}): 

\begin{align*}
  K^+_{\Rs,t}
  &= \bigcap_{\begin{smallmatrix}u\notin\Rs\\u\geq t\end{smallmatrix}} \Ker \rho^u_t
  = \bigcap_{\begin{smallmatrix}x\in\rcut^+\ \mathrm{or}\ y\in\tcut^+\\x\geq t_x\ \mathrm{and}\ y\geq t_y\end{smallmatrix}} \Ker \rho^{(x,y)}_t
  \stackrel{\mbox{\scriptsize (Eq. \ref{eq:weak_exactness})}}{=} \bigcap_{\begin{smallmatrix}x\in\rcut^+\ \mathrm{or}\ y\in\tcut^+\\x\geq t_x\ \mathrm{and}\ y\geq t_y\end{smallmatrix}} \Ker \rho^{(x,t_y)}_t + \Ker \rho^{(t_x,y)}_t \\
  &= \bigcap_{\begin{smallmatrix}x\in\rcut^+\\y\geq t_y\end{smallmatrix}} \Ker \rho^{(x,t_y)}_t + \Ker \rho^{(t_x,y)}_t  \ \cap\   \bigcap_{\begin{smallmatrix}y\in\tcut^+\\x\geq t_x\end{smallmatrix}} \Ker \rho^{(x,t_y)}_t + \Ker \rho^{(t_x,y)}_t\\
  &= \bigcap_{\begin{smallmatrix}x\in\rcut^+\\y= t_y\end{smallmatrix}} \Ker \rho^{(x,t_y)}_t + \Ker \rho^{(t_x,y)}_t   \ \cap\  \bigcap_{\begin{smallmatrix}y\in\tcut^+\\x= t_x\end{smallmatrix}} \Ker \rho^{(x,t_y)}_t + \Ker \rho^{(t_x,y)}_t\\
  &= \bigcap_{x\in\rcut^+} \Ker \rho^{(x,t_y)}_t   \ \cap\  \bigcap_{y\in\tcut^+} \Ker \rho^{(t_x,y)}_t
  = \Ker^+_{\rcut,t} \cap \Ker^+_{\tcut,t}.
\end{align*}

\bigskip
\noindent {\bf Proof of~(\ref{eq:equivKer-}):}
it is symmetric to that of~(\ref{eq:equivIma+}):

\begin{align*}
  K^-_{\Rs,t}
  &= \sum_{\begin{smallmatrix}u\in\Rs\\u\geq t\end{smallmatrix}} \Ker \rho^u_t
  = \sum_{\begin{smallmatrix}\rcut^-\ni x\geq t_x\\\tcut^-\ni y\geq t_y\end{smallmatrix}} \Ker \rho^{(x,y)}_t
  \stackrel{\mbox{\scriptsize (Eq. \ref{eq:weak_exactness})}}{=} \sum_{\begin{smallmatrix}\rcut^-\ni x\geq t_x\\\tcut^-\ni y\geq t_y\end{smallmatrix}} \Ker \rho^{(x,t_y)}_t + \Ker \rho^{(t_x,y)}_t \\
  &= \sum_{\rcut^-\ni x\geq t_x}\ \sum_{\tcut^-\ni y\geq t_y} \Ker \rho^{(x,t_y)}_t + \Ker \rho^{(t_x,y)}_t \\
  &= \sum_{\rcut^-\ni x\geq t_x} \left(\Ker \rho^{(x,t_y)}_t + \sum_{\tcut^-\ni y\geq t_y}  \Ker \rho^{(t_x,y)}_t\right) \\
  &= \left(\sum_{\rcut^-\ni x\geq t_x} \Ker \rho^{(x,t_y)}_t\right) + \left(\sum_{\tcut^-\ni y\geq t_y}  \Ker \rho^{(t_x,y)}_t\right)
  = \Ker^-_{\rcut,t} + \Ker^-_{\tcut,t}.
\end{align*}
%

\section{Proof of Theorem~\ref{thm:decomp_rect}}
\label{sec:proof_thm_decomp_rect}

  Write $\Rs=I\times J$, where $I,J$ are two intervals of the real line. The easiest situation is when both $I$ and $J$ are open. Then we can reparametrize each of them monotonously over the real line. The induced functor~$F$, going from the representation category of $\Rs$ to the representation category of $\sR^2$, is clearly an isomorphism of abelian categories. Moreover, the reparametrization preserves rectangles, therefore $F$ preserves the exactness property in addition to pointwise finite-dimensionality on objects. It follows that the $\sR^2$-indexed module $F(M)$ is both pfd and exact, and therefore block-decomposable by Theorem~\ref{thm_key}. The result for~$M$ follows.

  Now, in the general case, each of $I,J$ can be either open,
  right-open, left-open, or closed. In each case, monotonous
  reparametrizations as above allow us to reindex the module over some
  closed rectangle, possibly extending to infinity to the right, top,
  left, or bottom. From now on we assume without loss of generality
  that $\Rs$ is that closed rectangle. Define $\Rs^+$ to be the set
  $\{t\in\sR^2 \mid \exists s\in\Rs\ \mbox{with}\ t\leq s\}$. Note
  that $\Rs^+$ is the closed birth quadrant having the same lower-left
  corner as~$\Rs$.  Extend~$M$ to a module~$M^+$ indexed over~$\Rs^+$
  by taking its left Kan extension along the inclusion
  $\Rs\hookrightarrow \Rs^+$. Clearly, we have $M^+|_{\Rs}\simeq M$
  since the inclusion is a fully faithful functor between poset
  categories. Moreover, using the pointwise definition of the left Kan
  extension via colimits, at each point $t\in \Rs^+$ we have $M^+_t =
  \varinjlim M|_{\Rs\cap t^-}$, where $t^-$ denotes the set of points
  $s\in\sR^2$ such that $s\leq t$ (i.e. the closed death quadrant
  having $t$ as upper-right corner). Since $\Rs\cap t^-$ is a closed
  rectangle, it is a directed poset with a maximum element~$s$, so we have
  $M^+_t \simeq M_s$. As a result, $M^+$ is clearly
  pfd, and it is easily seen to be exact as well. Now we do a
  similar operation to extend~$M^+$ to a module~$M^{++}$ indexed
  over~$\sR^2$, that is, we take the right Kan extension of~$M^+$
  along the inclusion $\Rs^+\hookrightarrow\sR^2$. Again, by the fully
  faithful nature of the inclusion we have $M^{++}|_{\Rs^+} \simeq
  M^+$, and it follows from the pointwise definition of the right Kan
  extension via limits that~$M^{++}$ is both pfd and exact. Then,
  Theorem~\ref{thm_key} implies that $M^{++}$ is block-decomposable,
  from which we deduce by restriction that~$M$ itself is block-decomposable.

\section{Proof of Theorem~\ref{thm:decomp_U}}
\label{sec:proof_thm_decomp_U}

We follow the scheme of~\cite[Section~5]{botnan2018decomposition} but
do not make use of the general decomposition theorem in~\cite{botnan2018decomposition}. For any $s\leq t\in\sU$, let us call {\em $(s,t)$-square} the commutative diagram~(\ref{eq:quadrangle}), and say that it is {\em injective} if the first morphism in the exact sequence is injective, and {\em surjective} if the second morphism in the sequence is surjective.
Assume from now on that $M$ is an exact pfd bimodule indexed over~$\sU$.
%
\begin{lem}\label{lem:nontrivial_coker}
Suppose there exists some $(s,t)$-square that is not surjective. Then, we have $M\simeq M' \oplus \field_{\Fs\cap\sU}$ where $\Fs$ is a birth quadrant such that $s\notin\Fs$ and $t\in\Fs$. Similarly, if the $(s,t)$-square is not injective, then we have $M\simeq M'' \oplus \field_{\Fs\cap\sU}$ where $\Fs$ is a death quadrant such that $s\in\Fs$ and $t\notin\Fs$.  
\end{lem}
\begin{proof}
The proof is the same as the ones of Lemmas~5.12 and~5.13
in~\cite{botnan2018decomposition}, with the invocations of Lemmas~5.5
and~5.6 replaced by the one of our
Theorem~\ref{thm:decomp_rect}, observing that the only blocks $\Fs$ that can
make the $(s,t)$-square non-surjective (resp. non-injective) are the
birth quadrants $\Fs$ such that $s\notin\Fs\ni t$ (resp. death
quadrants $\Fs$ such that $s\in\Fs\not\ni t$). Note that the
conclusion of the proof of Lemma~5.13 is that there are morphisms
$\field_{\Fs\cap\sU} \hookrightarrow M \twoheadrightarrow
\field_{\Fs\cap\sU}$ composing to the identity
of~$\field_{\Fs\cap\sU}$, which implies that $M\simeq M' \oplus
\field_{\Fs\cap\sU}$. Meanwhile, the conclusion of the proof of
Lemma~5.12 is that there is an injective morphism
$\field_{\Fs\cap\sU}\hookrightarrow M$, which implies that
$M\simeq M'' \oplus \field_{\Fs\cap\sU}$ since $\field_{\Fs\cap\sU}$
is an injective module when~$\Fs$ is a death quadrant.
\end{proof}

\begin{lem}\label{lem:surjinj_invariance}
If an $(s,t)$-square is injective (resp. surjective) in~$M$, then so are all the $(s',t')$-squares for $s\leq s'\leq t'\leq t$.
\end{lem}
\begin{proof}
  We prove the injectivity part, the surjectivity part following by
  duality. First, observe that for any $t'$ such that $s\leq t'\leq t$
  we have $\Ker h_s^{(t'_x,s_y)} \cap \Ker v_s^{(s_x,t'_y)} \subseteq
  \Ker h_s^{(t_x,s_y)} \cap \Ker v_s^{(s_x,t_y)} = 0$, so we may
  assume without loss of generality that $t=t'$. Now, for any $s'$
  such that $s\leq s'\leq t$ we have the following commutative
  diagram, where $a=(s_x,t_y)$, $b=(s'_x,t_y)$, $c=(s_x, s'_y)$,
  $d=(t_x,s'_y)$, $e=(s'_x, s_y)$, $f=(t_x, s_y)$, and where every
  square is exact:
  \[\xymatrix{
    M_a \ar[r] &  M_b \ar[r] & M_t\\
    M_c \ar[u]\ar[r] & M_{s'} \ar[u]\ar[r] & M_d \ar[u]\\
    M_s \ar[u]\ar[r] & M_e \ar[u]\ar[r] & M_f \ar[u]
  } \]
  Assume for a contradiction that the $(s',t)$-square is not
  injective, which means that $\Ker h_{s'}^d \cap \Ker v_{s'}^b \neq
  0$. Take $\alpha \in (\Ker h_{s'}^d \cap \Ker v_{s'}^b) \setminus
  \{0\}$. Then, $0\in M_d$ has preimages $\alpha\in M_{s'}$ and $0\in
  M_f$, therefore the latter have a common preimage $\beta\in
  M_e$. Similarly, $\alpha\in M_{s'}$ and $0\in M_a$ have a common
  preimage $\gamma\in M_c$. Since $v_e^{s'}(\beta) = \alpha =
  h_c^{s'}(\gamma)$, $\beta$ and $\gamma$ have a common
  preimage~$\delta\in M_s$. This preimage cannot be zero since it is
  sent to~$\alpha\neq 0$ through $\rho_s^{s'}$, Meanwhile, it is sent
  to~$h_e^f(\beta) = 0\in M_f$ through $h_s^f$, and to $v_c^a(\gamma) =
  0\in M_a$ through $v_s^a$. Hence, it belongs to $\Ker h_s^f \cap
  \Ker v_s^a$, which implies that $\Ker h_s^f \cap \Ker v_s^a \neq 0$
  and thus contradicts the hypothesis that the $(s,t)$-square is
  injective.
  \end{proof}

\begin{lem}\label{lem:inj_surj}
  The module~$M$ decomposes as follows:
  \[
  M\simeq N \oplus \bigoplus_{\Fs\in \barcode} \field_{\Fs\cap\sU},
  \]
where $\barcode$ is a set of birth and death quadrants, and where every
$(s,t)$-square in~$N$ is both injective and surjective.
\end{lem}
\begin{proof}
Consider the sequence of points $(t_n)_{n\in\sN}$ defined by $t_n =
(n+1, n+1) \in\sU$. For each~$n$, take a finite zigzag
$s_n^0\leftarrow s_n^1\rightarrow s_n^2 \leftarrow \cdots \rightarrow
s_n^{r_n}$ as illustrated in Figure~\ref{fig:zz_death}, and call~$T_n$
the region bounded by~$t_n$ and this zigzag (in gray in the
figure).
\begin{figure}[htb]
  \centering
  \includegraphics[scale=0.7]{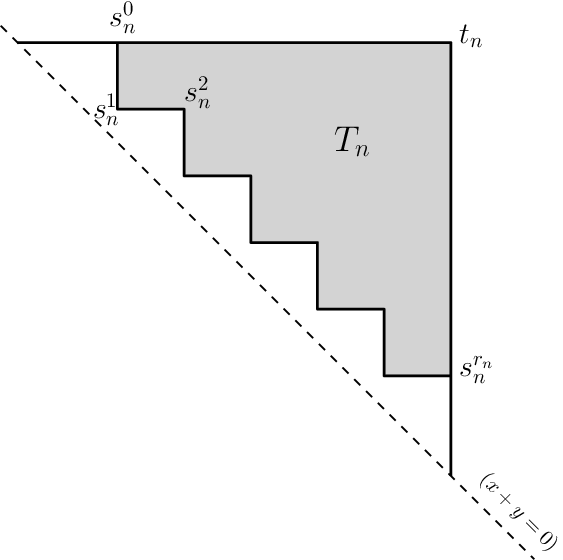}
  \caption{The region~$T_n$}
  \label{fig:zz_death}
\end{figure}
Furthermore, make sure that the zigzag converges to the
anti-diagonal $x+y=0$ as $n$ goes to infinity, in such a way that
$T_n\subseteq T_{n+1}$ for all $n\in\sN$.

Then, consider the sequence of submodules of~$M$ obtained  as follows:
\begin{itemize}
\item $M_0=M$;
\item For $n\geq 1$, while there exist internal summands of~$M_{n-1}$
  that are isomorphic to~$\field_{\Fs\cap\sU}$ for~$\Fs$ a birth
  (resp. death) quadrant with $s_n^1\notin\Fs\ni t_n$
  (resp. $s_n^1\in\Fs\not\ni t_n$), peel off such a summand
  from~$M_{n-1}$; then do the same for the $(s_n^3,t_n)$-square, and
  so on, until the $(s_{r_n-1},t_n)$-square. This procedure terminates
  because each square considered is finite-dimensional (hence finitely
  many internal summands are peeled off for that square) and there are
  finitely many such squares. Call $M_n$ the resulting submodule
  of~$M_{n-1}$.  Lemma~\ref{lem:nontrivial_coker} ensures that each
  square that has been considered is both injective and surjective
  in~$M_{n}$, and therefore so is every $(s',t')$-square with
  $s'\leq t' \in T_n$ by Lemma~\ref{lem:surjinj_invariance}.
\end{itemize}
We thus get a decreasing family $(M_n)_{n\in \sN}$ of submodules
of~$M$. Since~$M$ is pfd, this family converges pointwise to some
submodule~$N$. Every $(s,t)$-square in~$N$ is both injective and
surjective, because it belongs to~$T_n$ for some finite index~$n$ and
for all the indices beyond.
\end{proof}

To conclude the proof of Theorem~\ref{thm:decomp_U}, we must show that
the submodule~$N$ defined above is block-decomposable. The argument is
in fact the same as in the conclusion of the proof of Theorem~5.11
in~\cite{botnan2018decomposition}: using the prior that every square
is now both injective and surjective, we apply Corollary~5.9
from~\cite{botnan2018decomposition} (whose proof uses our main
Theorem~\ref{thm_key} and not the general decomposition theorem
in~\cite{botnan2018decomposition}) to get a direct-sum decomposition
of~$N$ into summands that are obtained from Kan extensions of interval
summands of a certain zigzag in the plane, each such extension giving
rise to a particular block module.

\bibliography{biblio} \bibliographystyle{plain}

\end{document}